\pdfoutput=1
\nonstopmode

\documentclass[leqno]{amsart}
\usepackage{calc}
\usepackage{amsmath, amssymb, amsthm, amscd}
\usepackage{url}
\usepackage{graphicx}
\usepackage{listings}
\usepackage{enumerate}
\usepackage{pdflscape}
\usepackage[english]{babel}
\usepackage{psfrag}
\usepackage[notightpage]{pst-pdf}
\usepackage{epsfig}
\usepackage[all]{xy}
\newdir{ >}{{}*!/-10pt/@{>}}
\newdir^{ (}{{}*!/-5pt/@^{(}}
\newdir_{ (}{{}*!/-5pt/@_{(}}
\def\cgaps#1{}
\def\Cgaps#1{}

\def\undersetbrace#1\to#2{\underbrace{#2}_{#1}}								 
\def\oversetbrace#1\to#2{\overbrace{#2}^{#1}}
\def\AMSunderset#1\to#2{\underset{#1}{#2}}
\def\AMSoverset#1\to#2{\overset{#1}{#2}}

\def\norm#1{\left\|{#1}\right\|}
\def\adj#1{\on{Adj}({#1})}

\newtheorem{ass}{Assumption}[subsection]
\newtheorem*{ass*}{Assumption}

\swapnumbers

\newtheorem*{prop*}{Proposition}

\newtheorem*{thm*}{Theorem}

\newtheorem*{lem*}{Lemma}

\newtheorem*{cor*}{Corollary}

\numberwithin{equation}{subsection}
\numberwithin{equation}{section}
\newenvironment{demo}[1]{{\textit{#1.}}}{\par\smallskip}

\def\ign#1{}             
\def\o{\circ}
\def\X{\mathfrak X}

\def\be{\beta}

\def\ep{\varepsilon}

\def\th{\theta}

\def\si{\sigma}

\def\ph{\varphi}

\def\om{\omega}
\def\Ga{\Gamma}
\def\De{\Delta}

\def\La{\Lambda}

\def\Om{\Omega}
\def\i{^{-1}}
\def\x{\times}
\def\p{\partial}
\let\on=\operatorname
\def\L{\mathcal L}
\def\grad{\on{grad}}%
\def\AMSonly#1{}

\def\Id{\on{Id}}
\def\R{\mathbb R}
\def\Tr{\on{Tr}}
\def\vol{{\on{vol}}}
\def\Vol{{\on{Vol}}}
\def\Imm{{\on{Imm}}}
\def\Emb{{\on{Emb}}}
\def\Diff{{\on{Diff}}}
\def\g{\overline{g}}
\def\grad{\on{grad}}
\def\dist{{\on{dist}}}
\def\Nor{{\on{Nor}}}
\def\sym{{\on{sym}}}
\def\alt{{\on{alt}}}

\def\Hor{{\on{Hor}}}
\def\hor{{\on{hor}}}
\def\vert{{\on{vert}}}

\begin{document}

\title{Sobolev metrics on shape space of surfaces}
\author{Martin Bauer, Philipp Harms, Peter W. Michor}
\address{
Martin Bauer, Peter W. Michor:
Fakult\"at f\"ur Mathematik, Universit\"at Wien,
Nordbergstrasse 15, A-1090 Wien, Austria.
}
\address{Philipp Harms:
EdLabs, Harvard University,
44 Brattle Street,
Cambridge, MA 02138, USA.}

\email{Bauer.Martin@univie.ac.at}
\email{pharms@edlabs.harvard.edu}
\email{Peter.Michor@univie.ac.at}
\dedicatory{Dedicated to Tudor Ratiu at the occasion of his 60th birthday}
\thanks{All authors were supported by 
FWF Project 21030 and by the 
   NSF-Focused Research Group: 
   The geometry, mechanics, and statistics of the infinite dimensional
   shape manifolds.}
\date{\today}
\keywords{}
\subjclass[2000]{Primary 58B20, 58D15, 58E12}
\begin{abstract}
Let $M$ and $N$ be connected manifolds without boundary with
$\dim(M) < \dim(N)$, and let $M$ compact. 
Then shape space in this work is either the manifold of submanifolds of $N$ that are 
diffeomorphic to $M$, 
or the orbifold of unparametrized immersions of $M$ in $N$.
We investigate the Sobolev Riemannian metrics on shape space: These are induced by 
metrics of the following form on the space of immersions:
$$ G^P_f(h,k) = \int_{M} \g( P^fh, k)\, \vol(f^*\g)$$
where $\g$ is some fixed metric on $N$, $f^*\g$ is the induced metric on $M$,
$h,k \in \Ga(f^*TN)$ are tangent vectors at $f$ to the space of embeddings or 
immersions, and $P^f$
is a positive, selfadjoint, bijective scalar pseudo differential operator of order $2p$ depending 
smoothly on $f$. We consider later specifically the operator 
$P^f=1 + A\De^p$, where $\De$ is the Bochner-Laplacian on $M$ induced by the metric $f^*\g$.
For these metrics we compute the geodesic equations both on the space of immersions and on shape 
space, and also the conserved momenta arising from the obvious symmetries. 
We also show that the geodesic equation is well-posed on spaces of immersions, and also on 
diffeomorphism groups. We give examples of numerical solutions.  
\end{abstract}

\maketitle

\section{Introduction}\label{in}

Many procedures in science, engineering, and medicine produce data in the form of shapes.
If one expects such a cloud to follow roughly a submanifold of a certain type, 
then it is of utmost importance to describe the space of all possible submanifolds 
of this type (we call it a shape space hereafter) 
and equip it with a significant metric which is able to distinguish special features 
of the shapes. Most of the metrics used today in data analysis and computer vision are of an ad-hoc 
and naive nature; one embeds shape space in some Hilbert space or Banach space and uses the 
distance therein. Shortest paths are then line segments, but they leave shape space quickly. 

Riemannian metrics on shape space itself are a better solution.
They lead to geodesics, to curvature and diffusion. 
Eventually one also needs statistics on shape space like means of clustered subsets of data 
(called Karcher means on Riemannian manifolds) and standard deviations. 
Here curvature will play an essential role; statistics on Riemannian 
manifolds seems hopelessly underdeveloped just now. 

\subsection{The shape spaces used in this work}

Thus, initially, by a shape we mean a
smoothly embedded surface in $N$ which is diffeomorphic to $M$. 
The space of these shapes will be  denoted $B_e=B_e(M,N)$ and viewed as the 
quotient (see \cite{Michor102} for more details)
$$ B_e(M,N) = \on{Emb}(M,N)/\on{Diff}(M)$$
of the open subset $\on{Emb}(M,N)\subset C^\infty(M,N)$ 
of smooth embeddings of $M$ in $N$,
modulo the group of smooth diffeomorphisms of $M$. 
It is natural to consider all possible {\it immersions} as well
as embeddings, and thus introduce the larger
space $B_i=B_i(M,N)$ as the quotient of the space of smooth immersions by the
group of diffeomorphisms of $M$ (which is, however, no longer a manifold,
but an orbifold with finite isotropy groups, see \cite{Michor102}).
\begin{equation*}
\xymatrix{
\on{Emb}(M,N) \ar@{^{(}->}[d] \ar@{->>}[r] &
\on{Emb}(M,N)/\on{Diff}(M) \ar@{^{(}->}[d] \ar@{=}[r]  & 
B_e(M,N) \ar@{^{(}->}[d] \\
\Imm(M,N) \ar@{->>}[r] &
\Imm(M,N)/\on{Diff}(M)  \ar@{=}[r]  & B_i(M,N)
}
\end{equation*}
More generally, a shape will be an element of the Cauchy completion (i.e., the metric completion 
for the geodesic distance) of $B_i(M,N)$ with 
respect to a suitably chosen Riemannian metric. This will allow for corners. 
In practice, discretization for numerical algorithms will hide the need to go to the Cauchy 
completion. 

\subsection{Where this work comes from}
In \cite{Michor107}, Michor and Mumford have investigated a variety of Riemannian metrics on the 
shape space 
$$
B_i(S^1,\mathbb R^2)=\Imm(S^1,\mathbb R^2)/\on{Diff}(S^1)
$$
of unparametrized immersion of the circle into the plane. 
In \cite[section~3.10]{Michor98} they found that the simplest such metric has vanishing geodesic distance; 
this is the metric induced by $L^2(\text{arc length})$ on $\Imm(S^1,\mathbb R^2)$:
\begin{align*}
G^0_f(h,k) &= \int_{S^1} \langle h(\th),k(\th) \rangle |f'(\th)|\,d\th , 
\\
f &\in \Imm(S^1,\mathbb R^2), \quad h,k \in C^\infty(S^1,\mathbb R^2) = 
T_f\Imm(S^1,\mathbb R^2).
\end{align*}
In \cite{Michor102} they  found that the vanishing geodesic distance phenomenon for the 
$L^2$-metric occurs also in 
the more general shape space $\Imm(M,N)/\on{Diff}(M)$ where $S^1$ is replaced by a compact 
manifold $M$ and Euclidean $\mathbb R^2$ is replaced by 
Riemannian manifold $N$; it also occurs on the full diffeomorphism group $\on{Diff}(N)$, but not on 
the subgroup $\on{Diff}(N,\vol)$ of volume preserving diffeomorphisms, where the geodesic 
equation for the $L^2$-metric is the Euler equation of an incompressible fluid. 
In \cite[sections 3, 4 and 5]{Michor107} three classes of metrics were investigated: Almost local metrics on planar curves, 
Sobolev metrics on planar curves, and metrics induced from Sobolev metrics on the diffeomorphism group of the plane. 
The results about almost local metrics from \cite[section~3]{Michor107} were generalized by the authors to the 
case of surfaces in \cite{Michor118}.

Now we take up the investigations from \cite[section~4]{Michor107}. The 
{\it immersion-Sobolev metric} 
considered there is
\begin{align*}
G^{\Imm,p}_f(h,k) &=
\int_{S^1} \big(\langle h,k \rangle + A.\langle D_s^p h,
  D_s^p k  \rangle \big). ds
\\&
= \int_{S^1}\langle L_p(h),k\rangle ds
 \qquad\text{  where }
\\
L_p(h) \text{ or } L_{p,f}(h) &= \big(I + (-1)^p A.D_s^{2p}\big)(h) \text{  and } 
D_s=\frac{\p_\th}{|f_\th|}.
\end{align*}
The interesting special case $p=1$ and $A \rightarrow \infty$ has been
studied in \cite{TrouveYounes2000, Younes1998} and 
in \cite{Michor111} where an isometry 
to an infinite dimensional Grassmannian with the Fubini-Study metric was described.  
In this case, the metric reduces to:
$$ G^{\Imm,1,\infty}_f(h,k) = \int_{S^1} \langle D_s(h), D_s(k)
\rangle.ds$$
The cases $p=1,2$ and $A \rightarrow \infty$ have also been treated in \cite{MennucciYezzi2008}, 
where estimates on the geodesic distance are proven and the metric completion of the space of curves 
is characterized. 

In this work we generalize the immersion-Sobolev metrics from \cite[section~4]{Michor107} 
to higher dimensions and to non-flat ambient space, namely to the shape space 
$B_i(M,N)= \Imm(M,N)/\on{Diff}(M)$ of surfaces of type $M$ in 
$N$; here $M$ is a compact orientable 
connected manifold of smaller dimension than $N$, for example a sphere $S^m, m<\dim(N)$.

\subsection{Riemannian metrics}\label{in:ri}
The tangent space $T_f \Imm(M,N)$ at an immersion $f$ consists 
of all vector fields along $f$: 
$$T_f \Imm(M,N) = \Ga(f^*TN) \cong \{h \in C^\infty(M,TN): \pi_{TN} \o h = f\}. $$
A Riemannian metric on $\Imm(M,N)$ is a family of
positive definite inner products $G_f(h,k)$ where 
$f \in \Imm(M,N)$ and $h,k \in T_f\Imm(M,N)$.
Each metric is {\it weak} in the sense that $G_f$, viewed as linear map from 
$T_f\Imm(M,N)$ into its dual consisting of distributional sections of $f^*TN$ is injective. 
(But it can never be surjective.)
We require that our
metrics will be invariant under the action of $\on{Diff}(M)$, hence the
quotient map dividing by this action will be a Riemannian submersion. 
This means that the tangent map of the quotient map $\Imm(M,N)\to B_i(M,N)$ is 
a metric quotient mapping between all tangent spaces.
Thus we will get Riemannian metrics on $B_i$.
For any $f\in \Imm(M,N)$ those vectors in 
$T_f\Imm(M,N)$ which are $G_f$-perpendicuar to the 
$\on{Diff}(M)$-orbit through $f$ are called 
\emph{horizontal} (with respect to $G$). 
They form the $G_f$-orthogonal space to the orbit. A priori we do not 
know that it is a complementary space. For the metrics considered in this work it will turn out to be 
a complement. 

The simplest inner product on the tangent bundle to 
$\Imm(M,N)$ is
$$ G^0_f(h,k) = \int_{M} \g(h, k) \, \vol(f^*\g),$$
where $\g=\langle\quad,\quad \rangle$ is the Euclidean inner product on $N$.
Since the volume form $\vol(f^*\g)$ reacts equivariantly to the action of the 
group $\Diff(M)$, this metric is invariant, 
and the map to the quotient $B_i$ is a Riemannian 
submersion for this metric. 
The $G^0$-horizontal vectors in $T_f\Imm(M,N)$ are just those vector fields along $f$ 
which are pointwise $\g$-normal to $f(M)$; we will call them {\it normal} fields. 

All of the metrics we will look at will be of the form (see section \ref{so}):
$$ G^P_f(h,k) = \int_{M} \g( P^f h, k)\, \vol(f^*\g)$$
where $P^f:T_f\Imm \to T_f\Imm$ is a positive bijective operator depending smoothly on $f$,
which is selfadjoint unbounded in the Hilbert space $T_f\Imm$ with inner product $G^0_f$.
We will assume that $P$ is in addition equivariant with respect to reparametrizations, i.e.
$$P^{f\circ\ph}=\ph^* \o P^f \o (\ph\i)^* = \ph^*(P^f)\qquad \text{for all }\ph\in\on{Diff}(M).$$
The $G^P$-horizontal vectors will be those 
$h\in T_f\on{Emb}(M,N)=C^\infty(M,N)$ such that $P^fh$ is normal.

The tangent map of the quotient map $\on{Emb}(M,N)\to B_i(M,N)$ is then an isometry 
when restricted to the horizontal spaces, just as in the finite dimensional situation. 
Riemannian submersions have a very nice effect on geodesics: the geodesics on
the quotient space $B_i$ are exactly the images of the horizontal geodesics on the top
space $\Imm$; by a horizontal geodesic we mean a geodesic whose tangent lies 
in the horizontal bundle. 
The induced metric is invariant under the action of $\on{Diff}(M)$ and therefore induces 
a unique metric on $B_i$. 
See for example \cite[section~1]{Michor118}.
Later in section \ref{la} we 
shall consider the special case 
$P^f=1+A\Delta^p$. 

\subsection{Inner versus outer metrics}

The metrics studied in this work are induced from $\Imm(M,N)$ on shape space.
One might call them \emph{inner metrics} since the differential operator governing the metric 
is defined intrinsic to $M$. Intuitively, these metrics can be seen as describing some elastic 
or viscous behaviour of the shape itself. 

In contrast to these metrics, there are also metrics induced from $\on{Diff}(N)$ on shape space. 
(The widely used LDDMM algorithm uses such a metric.)
The differential operator governing these metrics is defined on all of $N$, even outside of the shape. 
Intuitively, these metrics can be seen as describing some elastic or viscous behaviour of the
ambient space $N$ that gets deformed as the shape changes. 
One might call these metrics \emph{outer metrics}.

\subsection{Contributions of this work.}

\begin{itemize}
\item This work is the first to treat Sobolev inner metrics on spaces of immersed
surfaces and on higher dimensional shape spaces. 

\item It contains the first description of how the geodesic equation can be 
formulated in terms of gradients of the metric with respect to itself 
when the ambient space is not flat. 
To achieve this, a covariant derivative on some bundles over immersions is defined. 
This covariant derivative is induced from the Levi-Civita covariant derivative on ambient space. 

\item The geodesic equation is formulated in terms of this covariant derivative. 
Well-posedness of the geodesic equation is shown under some
regularity assumptions that are verified for Sobolev metrics. 
Well-posedness also follows for the geodesic equation on 
diffeomorphism groups, where this result has not yet been obtained in that full generality.

\item To derive the geodesic equation, a variational formula for the Laplacian operator is developed.
The variation is taken with respect to the metric on the manifold where the Laplacian is 
defined. This metric in turn depends on the immersion inducing it. 

\item It is shown that Sobolev inner metrics separate points in shape space when the 
order of the differential operator governing the metric is high enough.
(The metric needs to be as least as strong as the $H^1$-metric.)
Thus Sobolev inner metrics overcome the degeneracy of the $L^2$-metric. 

\item The path-length distance of Sobolev inner metrics is compared to the Fr\'echet distance. 
It would be desirable to bound F\'echet distance by some Sobolev distance. 
This however remains an open problem. 

\item Finally it is demonstrated in some examples that the geodesic equation 
for the $H^1$-metric on shape space of surfaces in $\R^3$ can be solved numerically. 
\end{itemize}
Big parts of this work can also be found, partly in more details, in the doctoral theses of Martin Bauer
\cite{Bauer2010} and Philipp Harms \cite{Harms2010}. 

\section{Content of this work}

This work progresses from a very general setting to a specific one in three steps. 
In the beginning, a framework for general inner metrics is developed. 
Then the general concepts carry over to more and more specific inner metrics. 
\begin{itemize}
\item First, shape space is endowed with a \emph{general inner metric}, i.e with a metric
that is induced from a metric on the space of immersions, but that is unspecified otherwise.
It is shown how various versions of the geodesic equation can be expressed using gradients of the 
metric with respect to itself and how conserved quantities arise from symmetries. 
(This is section~\ref{sh}.)
\item Then it is assumed that the inner metric is defined via an elliptic pseudo-differential operator. 
Such a metric will be called a \emph{Sobolev-type metric}. 
The geodesic equation is formulated in terms of the operator, and
existence of horizontal paths of immersions within each equivalence class of paths is proven. 
(This is section~\ref{so}.)
Then estimates on the path-length distance are derived. Most importantly it is shown that
when the operator involves high enough powers of the Laplacian, then the metric does not have 
the degeneracy of the $L^2$-metric. (This is section~\ref{ge}.)
\item Motivated by the previous results it is assumed that the elliptic pseudo-differential operator 
is given by the \emph{Laplacian} and powers of it.
Again, the geodesic equation is derived. The formulas that are obtained are ready to be implemented numerically. 
(This is section~\ref{la}.)
\end{itemize}
The remaining sections cover the following material: 
\begin{itemize}
\item Section~\ref{no} treats some \emph{differential geometry of surfaces} that is needed in
this work. It is also a good reference for the notation that is used.
The biggest emphasis is on a rigorous treatment of the covariant derivative. 
Some material like the adjoint covariant derivative is not found in standard text books. 
\item Section~\ref{va} contains formulas for the \emph{variation} of the metric, volume form, 
covariant derivative and Laplacian with respect to the immersion inducing them.
These formulas are used extensively later. 
\item Section~\ref{su} covers the special case of \emph{flat ambient space}. 
The geodesic equation is simplified and conserved momenta for the Euclidean motion group
are calculated. Sobolev-type metrics are compared to the Fr\'echet metric which is 
available in flat ambient space. 
\item Section~\ref{di} treats \emph{diffeomorphism groups} of compact manifolds as a 
special case of the theory that has been developed so far.
\item In section~\ref{nu} it is shown in some examples that the geodesic equation 
on shape space can be solved \emph{numerically}. 
\end{itemize}

\section[Notation]{Differential geometry of surfaces and notation}\label{no}

In this section the differential geometric tools 
that are needed to deal with immersed surfaces
are presented and developed. 
The most important point is a rigorous treatment of the covariant derivative 
and related concepts. 

The notation of \cite{MichorH} is used. 
Some of the definitions can also be found in \cite{Kobayashi1996a}.
A similar exposition in the same notation is \cite{Michor118}.

\subsection{Basic assumptions and conventions}\label{no:as}

\begin{ass*}
It is always assumed that $M$ and $N$ are connected manifolds 
of finite dimensions $m$ and $n$, respectively. 
Furthermore it is assumed that $M$ is compact, 
and that $N$ is endowed with a Riemannian metric $\g$.
\end{ass*}

In this work, \emph{immersions} of $M$ into $N$ will be treated, i.e. 
smooth functions $M \to N$ with injective tangent mapping at every point.
The set of all such immersions will be denoted by $\Imm(M,N)$.
It is clear that only the case $\dim(M) \leq \dim(N)$ is of interest 
since otherwise $\Imm(M,N)$ would be empty.

Immersions or paths of immersions are usually denoted by $f$. 
Vector fields on $\Imm(M,N)$ or tangent vectors with foot point $f$, i.e., vector fields along $f$, will be called $h,k,m$, for example. 
Subscripts like $f_t = \p_t f = \p f/\p t$ denote differentiation 
with respect to the indicated variable, but subscripts are also used to indicate the 
foot point of a tensor field.

\subsection{Tensor bundles and tensor fields}\label{no:te}

The \emph{tensor bundles}
\begin{equation*}\xymatrix{
T^r_s M \ar[d] & 
T^r_s M \otimes f^*TN \ar[d] \\
M & M 
}\end{equation*}
will be used. Here $T^r_sM$ denotes the bundle of 
$\left(\begin{smallmatrix}r\\s\end{smallmatrix}\right)$-tensors on $M$, i.e.
$$T^r_sM=\bigotimes^r TM \otimes \bigotimes^s T^*M,$$
and $f^*TN$ is the pullback of the bundle $TN$ via $f$, see \cite[section~17.5]{MichorH}. 
A \emph{tensor field} is a section of a tensor bundle. Generally, when $E$ is a bundle, 
the space of its sections will be denoted by $\Ga(E)$. 

To clarify the notation that will be used later, 
some examples of tensor bundles and tensor fields are given now.
$S^k T^*M = L^k_{\on{sym}}(TM; \R)$ and
$\La^k T^*M = L^k_{\on{alt}}(TM; \R)$
are the bundles of symmetric and alternating 
$\left(\begin{smallmatrix}0\\k\end{smallmatrix}\right)$-tensors, respectively. 
$\Om^k(M)=\Ga(\La^k T^*M)$ is the space of differential forms,  
$\X(M)=\Ga(TM)$ is the space of vector fields, and 
$$\Ga(f^*TN) \cong \big\{ h \in C^\infty(M,TN): \pi_N \o h = f \big\}$$ 
is the space of \emph{vector fields along $f$}.

\subsection{Metric on tensor spaces}\label{no:me}

Let $\g \in \Gamma(S^2_{>0} T^*N)$ denote a fixed Riemannian metric on $N$. 
The \emph{metric induced on $M$ by $f \in \Imm(M,N)$} is the pullback metric 
\begin{align*}
g=f^*\g \in \Gamma(S^2_{>0} T^*M), \qquad g(X,Y)=(f^*\g)(X,Y) = \g(Tf.X,Tf.Y),
\end{align*}
where $X,Y$ are vector fields on $M$.
The dependence of $g$ on the immersion $f$ should be kept in mind.
Let $$\flat = \check g: TM \to T^*M \quad \text{and} \quad \sharp=\check g\i: T^*M \to TM.$$ 
$g$ can be extended to the cotangent bundle $T^*M=T^0_1M$ by setting
$$g\i(\alpha,\beta)=g^0_1(\alpha,\beta)=\alpha(\beta^\sharp)$$
for $\alpha,\beta \in T^*M$. 
The product metric 
$$g^r_s = \bigotimes^r g \otimes \bigotimes^s g\i$$
extends $g$ to all tensor spaces $T^r_s M$, and 
$g^r_s \otimes \g$ yields a metric on $T^r_s M \otimes f^*TN$. 

\subsection{Traces}\label{no:tr}

The \emph{trace} contracts pairs of vectors and co-vectors in a tensor product: 
\begin{align*}
\Tr:\; T^*M \otimes TM = L(TM,TM) \to M \x \R
\end{align*}
A special case of this is the operator
$i_X$ inserting a vector $X$ into a co-vector or into a covariant factor of a tensor product.
The inverse of the metric $g$ can be used to define a trace 
$$\Tr^g: T^*M \otimes T^*M \to M \x \R$$
contracting pairs of co-vecors.
Note that $\Tr^g$ depends on the metric whereas $\Tr$ does not. 
The following lemma will be useful in many calculations: 

\begin{lem*}
\begin{equation*}
g^0_2(B,C)= \on{Tr}(g\i B g\i C) \quad \text{for $B,C \in T^0_2M$ if $B$ or $C$ is symmetric.}
\end{equation*}
(In the expression under the trace, $B$ and $C$ are seen as maps $TM \to T^*M$.)
\end{lem*}
\begin{proof}
Express everything in a local coordinate system $u^1, \ldots, u^{m}$ of $M$.
\begin{align*}
g^0_2(B,C)&=g^0_2\Big(\sum_{ik} B_{ik}du^i \otimes du^k,\sum_{jl}C_{jl}du^j \otimes du^l\Big) \\ & =
\sum_{ijkl} g^{ij}B_{ik}g^{kl}C_{jl} = \sum_{ijkl} g^{ji}B_{ik}g^{kl}C_{lj} 
= \on{Tr}(g\i B g\i C)
\end{align*}
Note that only the symmetry of $C$ has been used. 
\end{proof}

\subsection{Volume density}\label{no:vo}

Let $\Vol(M)$ be the \emph{density bundle} over $M$, see \cite[section~10.2]{MichorH}.
The \emph{volume density} on $M$ induced by $f \in \Imm(M,N)$ is 
$$\vol(g)=\vol(f^*\g) \in \Ga\big(\Vol(M)\big).$$
The \emph{volume} of the immersion is given by
$$\Vol(f)=\int_M \vol(f^*\g)=\int_M \vol(g).$$
The integral is well-defined since $M$ is compact. If $M$ is oriented the volume 
density may be identified with a differential form.

\subsection{Metric on tensor fields}\label{no:me2}

A \emph{metric on a space of tensor fields} is defined by integrating the appropriate metric on the 
tensor space with respect to the volume density:
$$\widetilde{g^r_s}(B,C)=\int_M g^r_s\big(B(x),C(x)\big)\vol(g)(x)$$
for $B,C \in \Gamma(T^r_sM)$, and
$$\widetilde{g^r_s \otimes \g}(B,C) = \int_M g^r_s\otimes \g \big(B(x),C(x)\big)\vol(g)(x)$$
for $B,C \in \Gamma(T^r_sM \otimes f^*TN)$, $f \in \Imm(M,N)$. 
The integrals are well-defined because $M$ is compact.

\subsection{Covariant derivative}\label{no:co}

Covariant derivatives on vector bundles as explained in \cite[sections 19.12, 22.9]{MichorH}
will be used.
Let $\nabla^g, \nabla^{\g}$ be the \emph{Levi-Civita covariant derivatives} on $(M,g)$
and $(N,\g)$, respectively. 
For any manifold $Q$ and vector field $X$ on $Q$, one has
\begin{align*}
\nabla^g_X:C^\infty(Q,TM) &\to C^\infty(Q,TM), & h &\mapsto \nabla^g_X h \\
\nabla^{\g}_X: C^\infty(Q,TN) &\to C^\infty(Q,TN), & h &\mapsto \nabla^{\g}_X h.
\end{align*}
Usually the symbol $\nabla$ will be used for all covariant derivatives.
It should be kept in mind that $\nabla^g$ depends on the metric $g=f^*\g$ and therefore also on 
the immersion $f$.
The following properties hold \cite[section~22.9]{MichorH}:
\begin{enumerate}
\item \label{no:co:ba}
$\nabla_X$ respects base points, i.e. 
$\pi \o \nabla_X h = \pi \o h$, where $\pi$ is the projection 
of the tangent space onto the base manifold. 
\item
$\nabla_X h$ is $C^\infty$-linear in $X$. So for a tangent vector $X_x \in T_xQ$, 
$\nabla_{X_x}h$ makes sense and equals $(\nabla_X h)(x)$.
\item
$\nabla_X h$ is $\R$-linear in $h$.
\item
$\nabla_X (a.h) = da(X).h + a.\nabla_X h$ for $a \in C^\infty(Q)$, the derivation property of $\nabla_X$.
\item \label{no:co:prop5}
For any manifold $\widetilde Q$ and smooth mapping 
$q:\widetilde Q \to Q$ and $Y_y \in T_y \widetilde Q$ one has
$\nabla_{Tq.Y_y}h=\nabla_{Y_y}(h \o q)$. If $Y \in \X(Q_1)$ and $X \in \X(Q)$ are $q$-related, then 
$\nabla_Y(h \o q) = (\nabla_X h) \o q$.
\end{enumerate}
The two covariant derivatives $\nabla^g_X$ and $\nabla^{\g}_X$ 
can be combined to yield a covariant derivative $\nabla_X$ acting on
$C^\infty(Q,T^r_sM \otimes TN)$ by additionally requiring the following properties 
\cite[section 22.12]{MichorH}:
\begin{enumerate}
\setcounter{enumi}{5} 
\item $\nabla_X$ respects the spaces $C^\infty(Q,T^r_sM \otimes TN)$. 
\item $\nabla_X(h \otimes k) = (\nabla_X h) \otimes k + h \otimes (\nabla_X k)$, 
a derivation with respect to the tensor product.
\item $\nabla_X$ commutes with any kind of contraction (see \cite[section 8.18]{MichorH}). 
A special case of this is
$$\nabla_X\big(\alpha(Y)\big)=(\nabla_X \alpha)(Y)+\alpha(\nabla_X Y) \quad 
\text{for } \alpha\otimes Y :N \to T^1_1M.$$
\end{enumerate}
Property \eqref{no:co:ba} is important because it implies that $\nabla_X$ 
respects spaces of sections of bundles. 
For example, for $Q=M$ and $f \in C^\infty(M,N)$, one gets
$$\nabla_X : \Ga(T^r_s M \otimes f^* TN) \to \Ga(T^r_s M \otimes f^* TN). $$

\subsection{Swapping covariant derivatives}\label{no:sw}

Some formulas allowing to swap covariant derivatives will be used repeatedly. 
Let $f$ be an immersion, $h$ a vector field along $f$ and $X,Y$ vector fields on $M$. 
Since $\nabla$ is torsion-free, one has \cite[section~22.10]{MichorH}:
\begin{equation}\label{no:sw:to}
\nabla_X Tf.Y-\nabla_Y Tf.X -Tf.[X,Y] = \on{Tor}(Tf.X,Tf.Y) = 0.
\end{equation}
Furthermore one has \cite[section~24.5]{MichorH}:
\begin{equation}\label{no:sw:r}
\nabla_X \nabla_Y h - \nabla_Y \nabla_X h - \nabla_{[X,Y]} h 
= R^{\g} \o (Tf.X,Tf.Y) h, 
\end{equation}
where $R^{\g} \in \Om^2\big(N;L(TN,TN)\big)$ is the Riemann curvature tensor of $(N,\g)$.

These formulas also hold when $f:\R \x M \to N$ is a path of immersions, 
$h:\R \x M \to TN$ is a vector field along $f$ and
the vector fields are vector fields on $\R \x M$. 
A case of special importance is when one of the vector fields is $(\p_t,0_M)$ and the 
other $(0_\R,Y)$, where $Y$ is a vector field on $M$. 
Since the Lie bracket of these vector fields vanishes, 
\eqref{no:sw:to} and \eqref{no:sw:r} yield
\begin{equation}\label{no:sw:to2}
\nabla_{(\p_t,0_M)} Tf.(0_{\R},Y)-\nabla_{(0_{\R},Y)} Tf.{(\p_t,0_M)} = 0
\end{equation}
and
\begin{equation}\label{no:sw:r2}
\nabla_{(\p_t,0_M)} \nabla_{(0_\R,Y)} h - \nabla_{(0_\R,Y)} \nabla_{(\p_t,0_M)} h
\\= R^{\g} \big(Tf.(\p_t,0_M),Tf.(0_\R,Y)\big) h .
\end{equation}

\subsection{Second and higher covariant derivatives}\label{no:co2}

When the covariant derivative is seen as a mapping
$$\nabla: \Gamma(T^r_s M) \to \Gamma(T^r_{s+1}M)\quad \text{or} \quad
\nabla : \Gamma(T^r_sM \otimes f^*TN) \to \Gamma(T^r_{s+1}M \otimes f^*TN),$$
then the \emph{second covariant derivative} is simply $\nabla\nabla=\nabla^2$.
Since the covariant derivative commutes with contractions,
$\nabla^2$ can be expressed as
$$\nabla^2_{X,Y} :=\iota_Y \iota_X \nabla^2 =
\iota_Y \nabla_X \nabla =
\nabla_X\nabla_Y -\nabla_{\nabla_XY} \qquad \text{for $X,Y\in \X(M)$.}$$
Higher covariant derivates are defined accordingly as $\nabla^k$, $k \geq 0$. 

\subsection{Adjoint of the covariant derivative}\label{no:co*}

The covariant derivative
$$\nabla: \Gamma(T^r_sM) \to  \Gamma( T^r_{s+1}M)$$
admits an \emph{adjoint} 
$$\nabla^*:\Gamma( T^r_{s+1}M)\to \Gamma(T^r_sM)$$
with respect to the metric $\widetilde{g}$, i.e.: 
$$\widetilde{g^r_{s+1}}(\nabla B, C)= \widetilde{g^r_s}(B, \nabla^* C).$$
In the same way, $\nabla^*$ can be defined when $\nabla$ is acting on 
$\Ga(T^r_s M \otimes f^*TN)$.
In either case it is given by
$$\nabla^*B=-\on{Tr}^g(\nabla B), $$
where the trace is contracting the first two tensor slots of $\nabla B$. 
This formula will be proven now: 

\begin{proof}
The result holds for decomposable tensor fields $\be \otimes B \in \Ga(T^r_{s+1} M)$ since
\begin{align*} &
\widetilde {g^r_s}\Big(\nabla^*(\be \otimes B),C\Big) = 
\widetilde {g^r_{s+1}}\Big(\be \otimes B,\nabla C\Big) = 
\widetilde {g^r_{s}}\Big(B,\nabla_{\be^\sharp} C\Big) \\&\qquad= 
\int_M \L_{\be^\sharp} g^r_{s}(B, C) \vol(g) - \int_M g^r_s(\nabla_{\be^\sharp} B,C) \vol(g) 
\\&\qquad= 
\int_M -g^r_{s}(B, C) \L_{\be^\sharp} \vol(g) - \int_M g^r_s\big(\Tr^g(\be \otimes \nabla B),C\big) \vol(g) 
\\&\qquad= 
\widetilde {g^r_{s}}\Big(-\on{div}(\be^\sharp) B - \Tr^g(\be \otimes \nabla B),C\Big) \\&\qquad= 
\widetilde {g^r_{s}}\Big(-\on{div}(\be^\sharp) B + \Tr^g((\nabla\be) \otimes B) 
-\Tr^g(\nabla (\be \otimes B))  ,C\Big)
\\&\qquad= 
\widetilde {g^r_{s}}\Big(-\on{div}(\be^\sharp) B + \Tr^g(\nabla\be) B 
-\Tr^g(\nabla (\be \otimes B))  ,C\Big)\\&\qquad= 
\widetilde {g^r_{s}}\Big(0-\Tr^g(\nabla (\be \otimes B)), C\Big)
\end{align*}
Here it has been used that $\nabla_X g=0$, 
that $\nabla_{X}$ commutes with any kind of contraction and acts as a derivation 
on tensor products \cite[section~22.12]{MichorH}
and that $\on{div}(X) = \Tr(\nabla X)$ for all vector fields $X$ \cite[section~25.12]{MichorH}.
To prove the result for $\be \otimes B \in \Ga(T^r_{s+1} M \otimes f^*TN)$ one simply has
to replace $g^r_{s}$ by $g^r_{s} \otimes \g$. 
\end{proof}

\subsection{Laplacian}\label{no:la}

The definition of the Laplacian used in this work is the \emph{Bochner-Laplacian}. 
It can act on all tensor fields $B$ and is defined as
$$\Delta B = \nabla^*\nabla B = - \on{Tr}^g(\nabla^2 B).$$

\subsection{Normal bundle}\label{no:no}

The \emph{normal bundle} $\Nor(f)$ of an immersion $f$ is a sub-bundle of $f^*TN$ 
whose fibers consist of all vectors that are orthogonal to the image of $f$:
$$\Nor(f)_x = \big\{ Y \in T_{f(x)}N : \forall X \in T_xM : \g(Y,Tf.X)=0  \big\}.$$
If $\dim(M)=\dim(N)$ then the fibers of the normal bundle are but the zero vector. 
Any vector field $h$ along $f \in \Imm$ can be decomposed uniquely 
into parts {\it tangential} and {\it normal} to $f$ as
$$h=Tf.h^\top + h^\bot,$$ 
where $h^\top$ is a vector field on $M$ and $h^\bot$ is a section of the normal bundle $\Nor(f)$.

\subsection{Second fundamental form and Weingarten mapping}\label{no:we}

Let $X$ and $Y$ be vector fields on $M$. 
Then the covariant derivative $\nabla_X Tf.Y$ splits into tangential and a normal parts as
$$\nabla_X Tf.Y=Tf.(\nabla_X Tf.Y)^\top + (\nabla_X Tf.Y)^\bot = Tf.\nabla_X Y + S(X,Y).$$
$S$ is the \emph{second fundamental form of $f$}. 
It is a symmetric bilinear form with values in the normal bundle of $f$. 
When $Tf$ is seen as a section of $T^*M \otimes f^*TN$ one has $S=\nabla Tf$ since
$$S(X,Y) = \nabla_X Tf.Y - Tf.\nabla_X Y = (\nabla Tf)(X,Y).$$
The trace of $S$ is the \emph{vector valued mean curvature} $\Tr^g(S) \in \Ga\big(\Nor(f)\big)$.

\section{Shape space}\label{sh}

Briefly said, in this work the word shape means an \emph{unparametrized surface}.
(The term surface is used regardless of whether it has dimension two or not.)  
This section is about the infinite dimensional space of all shapes. 
First some spaces of parametrized and unparametrized surfaces are described, 
and it is shown how to define Riemannian metrics on them. 
The geodesic equation and conserved quantities arising from symmetries are derived.

The agenda that is set out in this section
will be pursued in section~\ref{so} when the arbitrary metric is replaced
by a Sobolev-type metric involving a pseudo-differential operator and later in 
section~\ref{la} when the pseudo-differential operator is replaced by
an operator involving powers of the Laplacian.

\subsection{Riemannian metrics on immersions}\label{sh:im}\label{sh:na}\label{sh:ri}

The space of smooth immersions of the manifold $M$ into the manifold $N$ 
will be denoted by $\Imm(M,N)$ or briefly $\Imm$. It is a smooth Fr\'echet manifold containing
the space $\Emb(M,N)$ of embeddings of $M$ into $N$ as an open subset \cite[theorem~44.1]{MichorG}.
Consider the following \emph{natural bundles of $k$-multilinear mappings}: 
\begin{equation*}\xymatrix{
L^k(T\Imm;\R) \ar[d] & L^k(T\Imm;T\Imm ) \ar[d] \\
\Imm & \Imm
}\end{equation*}
These bundles are isomorphic to the bundles
\begin{equation*}\xymatrix{
L\left(\widehat\bigotimes^k T\Imm;\R\right)\ar[d] & 
L\left(\widehat\bigotimes^k T\Imm;T\Imm\right)\ar[d]\\
\Imm & \Imm
}\end{equation*}
where $\widehat\bigotimes$ denotes the $c^\infty$-completed bornological tensor product of 
locally convex vector spaces \cite[section~5.7, section~4.29]{MichorG}.
Note that $L(T\Imm;T\Imm)$ is not isomorphic to 
$T^*\Imm \;\widehat\otimes\; T\Imm$ 
since the latter bundle corresponds to multilinear mappings with finite rank.

It is worth to write down more explicitly what some of these bundles of multilinear mappings are.  
The \emph{tangent space to $\Imm$} is given by
\begin{align*}
T_f\Imm &= C^\infty_f(M,TN) := \big\{ h \in C^\infty(M,TN): \pi_N \o h =f\big\}, \\
T\Imm &= C^\infty_{\Imm}(M,TN) := \big\{ h \in C^\infty(M,TN): \pi_N \o h \in \Imm \big\}.
\end{align*}
Thus $T_f\Imm$ is the space of vector fields along the immersion $f$.
Now the \emph{cotangent space to $\Imm$} will be described.
The symbol $\widehat\otimes_{C^\infty(M)}$ means that the tensor product
is taken over the algebra $C^\infty(M)$.
\begin{align*}
T^*_f\Imm &= L(T_f\Imm;\R) = C^\infty_f(M,TN)' = 
C^\infty(M)'\; \widehat\otimes_{C^\infty(M)} C^\infty_f(M,T^*N) \\
T^*\Imm &= L(T\Imm;\R) = C^\infty(M)'\; \widehat\otimes_{C^\infty(M)} C^\infty_{\Imm}(M,T^*N) 
\end{align*}
The bundle $L^2_{\on{sym}}(T\Imm;\R)$ is of interest for the definition of a Riemannian 
metric on $\Imm$. 
(The subscripts $_\sym$ and $_\alt$ indicate symmetric and alternating multilinear maps, respectively.) 
Letting $\otimes_S$ denotes the symmetric tensor product and 
$\widehat\otimes_S$ the $c^\infty$-completed bornological symmetric tensor product, one has
\begin{align*}
L^2_{\on{sym}}(T_f\Imm;\R) &= (T_f\Imm\; \widehat\otimes_S\; T_f\Imm)' = 
\big(C^\infty_f(M,TN) \; \widehat\otimes_S\; C^\infty_f(M,TN)\big)' \\ &=
\big(C^\infty_f(M,TN \; \otimes_S\; TN) \big)' \\&=
C^\infty(M)' \;\widehat\otimes_{C^\infty(M)} C^\infty_f(M,T^*N \; \otimes_S\; T^*N) 
\\
L^2_{\on{sym}}(T\Imm;\R) &= 
C^\infty(M)' \;\widehat\otimes_{C^\infty(M)} C^\infty_{\Imm}(M,T^*N \;\otimes_S\; T^*N) 
\end{align*}
A \emph{Riemannian metric $G$ on $\Imm$} is a section of the bundle
$L^2_{\on{sym}}(T\Imm;\R)$
such that at every $f \in \Imm$, $G_f$ is a symmetric positive definite bilinear mapping 
$$G_f: T_f\Imm \x T_f\Imm \to \R.$$
Each metric is {\it weak} in the sense that $G_f$, seen as a mapping
$$G_f: T_f\Imm \to T^*_f\Imm$$
is injective. (But it can never be surjective.)

\subsection{Covariant derivative $\nabla^{\g}$ on immersions}\label{sh:cov}

The covariant derivative $\nabla^{\g}$ defined in section~\ref{no:co} induces a
\emph{covariant derivative over immersions} as follows.
Let $Q$ be a smooth manifold. Then one identifies
\begin{align*}
&h \in  C^\infty\big(Q,T\Imm(M,N)\big) && \text{and} && X \in \X(Q)
\intertext{with}
&h^{\wedge} \in C^\infty(Q \x M, TN) && \text{and} && (X,0_M) \in \X(Q \x M).
\end{align*}
As described in section~\ref{no:co} one has the covariant derivative
$$\nabla^{\g}_{(X,0_M)} h^{\wedge} \in C^\infty\big(Q \x M, TN).$$
Thus one can define
$$\nabla_X h = \left(\nabla^{\g}_{(X,0_M)} h^{\wedge}\right)^{\vee} \in C^\infty\big(Q,T\Imm(M,N)\big).$$
This covariant derivative is torsion-free by section~\ref{no:sw}, formula~\eqref{no:sw:to}. 
It respects the metric $\g$ but in general does not respect $G$. 

It is helpful to point out some special cases of how this construction can be used. 
The case $Q=\R$ will be important to formulate the geodesic equation. 
The expression that will be of interest in the formulation of the 
geodesic equation is $\nabla_{\p_t} f_t$, which is 
well-defined when $f:\R \to \Imm$ is a path of immersions and $f_t: \R \to T\Imm$ is its velocity. 

Another case of interest is $Q = \Imm$. Let $h, k, m \in \X(\Imm)$. Then the covariant 
derivative $\nabla_m h$ is well-defined and tensorial in $m$.
Requiring $\nabla_m$ to respect the grading of the spaces of multilinear maps, to act as a derivation 
on products and to commute with compositions of multilinear maps, one obtains 
as in section~\ref{no:co} a covariant 
derivative $\nabla_m$ acting on all mappings into 
the natural bundles of multilinear mappings over $\Imm$.
In particular, $\nabla_m P$ and $\nabla_m G$ are well-defined for 
\begin{align*}
P \in \Ga\big(L(T\Imm;T\Imm)\big), \quad
G \in \Ga\big(L^2_{\on{sym}}(T\Imm;\R)\big)
\end{align*}
by the usual formulas
\begin{align*}
(\nabla_m P)(h) &= \nabla_m\big(P(h)\big) - P(\nabla_mh), \\
(\nabla_m G)(h,k) &= \nabla_m\big(G(h,k)) - G(\nabla_m h,k) - G(h,\nabla_m k). 
\end{align*}

\subsection{Metric gradients}\label{sh:me}

The \emph{metric gradients} $H,K \in \Ga\big(L^2(T\Imm;T\Imm)\big)$ are uniquely defined by the equation
$$(\nabla_m G)(h,k)=G\big(K(h,m),k\big)=G\big(m,H(h,k)\big),$$
where $h,k,m$ are vector fields on $\Imm$ and 
the covariant derivative of the metric tensor $G$ is defined as in the previous section. 
(This is a generalization of the definition used in \cite{Michor107} 
that allows for a curved ambient space $N \neq \R^n$.)

Existence of $H, K$ has to proven case by case for each metric $G$, 
usually by partial integration. 
For Sobolev metrics, this will be proven in sections~\ref{la:ad} and \ref{la:ge}.

\begin{ass*}
Nevertheless it will be assumed for now that the metric gradients $H,K$ exist.
\end{ass*}

\subsection{Geodesic equation on immersions}\label{sh:ge}

\begin{thm*}
Given $H,K$ as defined in the previous section and $\nabla$ as defined in 
section~\ref{sh:cov}, the geodesic equation reads as
$$\nabla_{\p_t} f_t=\frac12 H_f(f_t,f_t)-K_f(f_t,f_t).$$
\end{thm*}
This is the same result as in \cite[section~2.4]{Michor107}, but in a more general setting.
\begin{proof}
Let $f: (-\ep,\ep) \x [0,1] \x M \to N$ be a one-parameter 
family of curves of immersions with fixed endpoints. 
The variational parameter will be denoted by $s \in (-\ep,\ep)$ and 
the time-parameter by $t \in [0,1]$. 
In the following calculation, let $G_f$ denote $G$ composed with $f$, i.e.
$$G_f: \R \to \Imm \to L^2_{\sym}(T\Imm;\R).$$ 
Remember that the covariant derivative on $\Imm$ that has been introduced in 
section~\ref{sh:cov} is torsion-free so that one has
$$\nabla_{\p_t}f_s - \nabla_{\p_s}f_t=Tf.[\p_t,\p_s]+\on{Tor}(f_t,f_s) = 0.$$
Thus the first variation of the energy of the curves is
\begin{align*}
\p_s \frac12 \int_0^1 G_f(f_t, f_t) dt &= 
\frac12 \int_0^1 (\nabla_{\p_s} G_f)(f_t, f_t)
+ \int_0^1 G_f(\nabla_{\p_s} f_t, f_t) dt 
\\&= 
\frac12 \int_0^1 (\nabla_{f_s} G)(f_t, f_t) 
+ \int_0^1 G_f(\nabla_{\p_t} f_s, f_t) dt 
\\&= 
\frac12 \int_0^1 (\nabla_{f_s} G)(f_t, f_t) dt
+ \int_0^1 \p_t\ G_f(f_s,f_t) dt \\&\qquad
- \int_0^1 (\nabla_{f_t} G)(f_s, f_t) dt 
- \int_0^1 G_f(f_s,\nabla_{\p_t} f_t) dt 
\\&=
\int_0^1 G\Big(f_s,\frac12 H(f_t,f_t)+0-K(f_t,f_t)-\nabla_{\p_t} f_t\Big) dt.
\end{align*}
If $f(0,\cdot,\cdot)$ is energy-minimizing, then one has at $s=0$ that
\begin{equation*}
\frac12 H(f_t,f_t)-K(f_t,f_t)-\nabla_{\p_t} f_t =0. \qedhere
\end{equation*}
\end{proof}

\subsection[Geodesic equation on immersions]%
{Geodesic equation on immersions in terms of the momentum}\label{sh:gemo}

In the previous section the geodesic equation for the velocity $f_t$ has been derived. 
In many applications it is more convenient to formulate the geodesic equation as an equation 
for the momentum $G(f_t,\cdot) \in T^*_f\Imm$. 
$G(f_t,\cdot)$ is an element of the \emph{smooth 
cotangent bundle}, also called \emph{smooth dual}, which is given by 
$$G(T\Imm) := \coprod_{f \in \Imm} G_f(T_f\Imm) = 
\coprod_{f \in \Imm} \{ G_f(h,\cdot): h \in T_f\Imm\} \subset T^*\Imm. $$
It is strictly smaller than $T^*\Imm$ since at every $f \in \Imm$
the metric $G_f: T_f\Imm \to T^*_f\Imm$ is injective but not surjective. It is called 
smooth since it does not contain distributional sections of $f^*TN$, whereas $T_f^*\Imm$ does. 

\begin{thm*}
The geodesic equation for the momentum $p \in T^*\Imm$ is given by
\begin{equation*}
\left\{\begin{aligned}
p &= G(f_t, \cdot) \\
\nabla_{\p_t} p &= \frac12 G_f\big( H(f_t,f_t),\cdot\big),
\end{aligned}\right.
\end{equation*}
where $H$ is the metric gradient defined in section~\ref{sh:me}
and $\nabla$ is the covariant derivative action on mappings into $T^*\Imm$ as 
defined in section~\ref{sh:cov}.
\end{thm*}

\begin{proof}
Let $G_f$ denote $G$ composed with the path $f:\R\to\Imm$, i.e.
$$G_f: \R \to \Imm \to L^2_{\sym}(T\Imm;\R).$$
Then one has
\begin{align*}
\nabla_{\p_t} p &= \nabla_{\p_t} \big(G_f(f_t,\cdot)\big) =
(\nabla_{\p_t}G_f)(f_t,\cdot) + G_f(\nabla_{\p_t} f_t,\cdot) \\&=
(\nabla_{f_t}G)(f_t,\cdot) + G_f\Big(\frac12 H(f_t,f_t) - K(f_t,f_t),\cdot\Big) \\&=
G_f\big(K(f_t,f_t),\cdot\big)+ G_f\Big(\frac12 H(f_t,f_t) - K(f_t,f_t),\cdot\Big) \qedhere
\end{align*}
\end{proof}
This equation is equivalent to \emph{Hamilton's equation} restricted to the smooth cotangent bundle:
\begin{equation*}
\left\{\begin{aligned}
p &= G(f_t, \cdot) \\
p_t &= (\grad^{\om} E)(p).
\end{aligned}\right.
\end{equation*} 
Here $\om$ denotes the restriction of the canonical symplectic form on $T^*\Imm$ to the smooth 
cotangent bundle and $E$ is the Hamiltonian
$$E: G(T\Imm) \to \R, \quad E(p) = G\i(p,p)$$
which is only defined on the smooth cotangent bundle.

\subsection{Shape space}\label{sh:sh}

$\Diff(M)$ acts smoothly on $\Imm(M,N)$ and $\Emb(M,N)$ 
by composition from the right. 
For $\Imm$, the action is given by the mapping
$$\Imm(M,N) \x \Diff(M) \to \Imm(M,N), \qquad (f,\ph) \mapsto r(f,\ph) = r^\ph(f)= f \o \ph.$$
The tangent prolongation of this group action is given by the mapping
$$T\Imm(M,N) \x \Diff(M) \to T\Imm(M,N), \qquad (h,\ph) \mapsto Tr^\ph(h) = h \o \ph.$$
\emph{Shape space} is defined as the orbit space with respect to this action. 
That means that in shape space, two mappings differing only in their parametrization
will be regarded the same. 

\begin{thm*}
Let $M$ be compact and of dimension $\leq n$.
Then $\Emb(M,N)$ is the total space of a smooth principal fiber bundle 
with structure group $\Diff(M)$, whose base manifold 
is a Hausdorff smooth Fr\'echet manifold denoted by
$$B_{e}(M,N) = \Emb(M,N)/\Diff(M).$$
However, the space
$$B_i(M,N) = \Imm(M,N)/\Diff(M)$$
is not a smooth manifold, but has singularities of orbifold type: 
Locally, it looks like a finite dimensional orbifold times an infinite dimensional 
Fr\'echet space. 
\end{thm*}

The proof for immersions can be found in \cite{Michor40} and 
the one for embeddings in \cite[section~44.1]{MichorG}. 
As with immersions and embeddings, the notation $B_i, B_e$ will be used
when it is clear that $M$ and $N$ are the domain and target of the mappings.

\subsection{Riemannian metrics on shape space}\label{sh:rish}

We start with a metric $G$ on $\Imm$. 
The mapping $\pi:\Imm \rightarrow B_i$ is a submersion of smooth manifolds, 
that is, $T\pi:T\Imm \rightarrow TB_i$ is surjective.
$$V=V(\pi):=\on{ker}(T\pi) \subset T\Imm$$
is called the {\it vertical subbundle}. 
The {\it horizontal subbundle}
is the $G$-orthogonal subspace of $V$: 
$$\Hor=\Hor(\pi,G):=V(\pi)^\bot \subset T\Imm.$$
It need not be a complement to $V$ (recall that the metric is weak; the complement could be in a 
suitable completion of the tangent space).
For all metrics in this paper it will turn out to be a complement, however. 
Then any vector $h \in T\Imm$ can be decomposed uniquely in vertical and horizontal components as
$$h=h^{\on{ver}}+h^{\hor}.$$
This definition extends to the cotangent bundle as follows: 
An element of $T^*\Imm$ is called horizontal when it annihilates all vertical vectors, 
and vertical when it annihilates all horizontal vectors. 

In the setting described so far, the mapping 
\begin{equation*}
T_f \pi|_{\Hor_f}:\Hor_f\rightarrow T_{\pi(f)}B_i
\end{equation*}
is an isomorphism of vector spaces for all $f\in \Imm$. 
This isomorphism will be used to describe the tangent space to $B_i$. 
If both $\Imm$ and $B_i$ are Riemannian manifolds and
if this isomorphism is also an isometry for all $f\in \Imm$, 
then $\pi$ is called a {\it Riemannian submersion}.
In that case, the metric $G$ on $\Imm$ is $\Diff(M)$-invariant.
This means that $G=(r^\ph)^* G$ for all $\ph \in \Diff(M)$, where 
$r^\ph$ denotes the right action of $\ph$ on $\Imm$ that was described in section~\ref{sh:sh}. 
This condition can be spelled out in more details using the definition of $r^\ph$ as follows:
\begin{align*}
G_f(h,k)=\big((r^\ph)^* G\big)(h,k)
=G_{r^\ph(f)}\big(Tr^\ph(h),Tr^\ph(k)\big)
=G_{f \o \ph}(h \o \ph,k \o \ph). 
\end{align*}
The following theorem establishes the converse statement: 
\begin{thm*}
Given a  $\Diff(M)$-invariant Riemannian metric on $\Imm$, 
there is a unique Riemannian metric on the quotient space $B_i$ such that the
quotient map $\pi:\Imm \to B_i$ is a Riemannian submersion. 
\end{thm*}

\begin{proof}
If the horizontal bundle $\Hor_f$ is a complement to $V_f$ then $T_f\pi:\Hor_f \to T_{\pi(f)}B_i$ 
is an isomorphism (off the orbifold singularities of $B_i$) and we can induce the metric on 
$T_{\pi(f)}B_i$ which is independent of the choice of $f$ in the fiber over $\pi(f)$ by the the 
$\Diff(M)$-invariance of the metric. If it is not a complement one has to consider the metric 
quotient norm. See for example \cite[section 3]{Michor98}.
\end{proof}

\begin{ass*}
It will always be assumed that a $\Diff(M)$-invariant metric $G$ on $\Imm(M,N)$ is given 
and that shape space $B_i$ is endowed with the unique metric such that
the quotient map is a Riemannian submersion.
\end{ass*}

\subsection{Riemannian submersions and geodesics}\label{sh:sub}

It follows from the general theory of Riemannian submersions that horizontal geodesics 
in the top space correspond nicely to geodesics in the quotient space: 
\begin{thm*}
Let $c:[0,1]\rightarrow \Imm$ be a geodesic.
\begin{enumerate}
\item If $c'(t)$ is horizontal at one $t$, then it is horizontal at all $t$. 
\item If $c'(t)$ is horizontal then $\pi \circ c$ is a geodesic in $B_i$.
\item If every curve in $B_i$ can be lifted to a horizontal curve in $\Imm$, 
then there is a one-to-one correspondence between curves in $B_i$ and horizontal curves in $\Imm$. 
This implies that instead of solving the geodesic equation on $B_i$ one can equivalently solve
the equation for horizontal geodesics in $\Imm$.
\end{enumerate}
\end{thm*}
See \cite[section~26]{MichorH} for the proof.

\subsection{Geodesic equation on shape space}\label{sh:gesh}

Theorem~\ref{sh:sub} applied to the Riemannian submersion $\pi: \Imm \to B_i$ yields: 
\begin{thm*}
Assuming that every curve in $B_i$ can be lifted to a horizontal curve in $\Imm$, 
the geodesic equation on shape space is equivalent to
\begin{equation}\label{sh:gesh:eq1}
\left\{\begin{aligned}
f_t&=f_t^{\hor}\in \Hor \\
(\nabla_{\p_t}f_t)^{\hor} &= \Big(\frac12 H(f_t,f_t)-K(f_t,f_t)\Big)^{\hor},
\end{aligned}\right.
\end{equation}
where $f$ is a horizontal curve in $\Imm$, where $H,K$ are the metric gradients
defined in section~\ref{sh:me} and where $\nabla$ is the covariant derivative defined
in section~\ref{sh:cov}.
\end{thm*}
This is a consequence of the $\Diff(M)$-invariance of the metric $G$ 
and the conservation of the reaparametrization momentum. 
A general proof can be found in \cite[section~3.14]{Harms2010}.

It will be shown in section~\ref{so:ho2} that curves in $B_i$ 
can be lifted to horizontal curves in $\Imm$ for the very general class of 
Sobolev type metrics. Thus all assumptions and conclusions of the theorem hold. 

\subsection[Geodesic equation on shape space]%
{Geodesic equation on shape space in terms of the momentum}\label{sh:geshmo}

As in the previous section, theorem~\ref{sh:sub} will be applied to the Riemannian submersion 
$\pi: \Imm \to B_i$. 
But this time, the formulation of the geodesic equation in terms of the momentum will be used, 
see section~\ref{sh:gemo}.
As will be seen in section~\ref{so:geshmo}, this is the most convenient formulation
of the geodesic equation for Sobolev-type metrics.
\begin{thm*}
Assuming that every curve in $B_i$ can be lifted to a horizontal curve in $\Imm$, 
the geodesic equation on shape space is equivalent to the set of equations
\begin{equation*}
\left\{\begin{aligned}
p &= G_f(f_t,\cdot) \in \Hor \subset T^*\Imm, \\
(\nabla_{\p_t}p)^{\hor} &= \frac12 G_f\big(H(f_t,f_t),\cdot)^{\hor}.
\end{aligned}\right.
\end{equation*}
Here $f$ is a curve in $\Imm$, 
$H$ is the metric gradient defined in section~\ref{sh:me},
and $\nabla$ is the covariant derivative defined in section~\ref{sh:cov}.
$f$ is horizontal because $p$ is horizontal. 
\end{thm*}

\section{Variational formulas}\label{va}

Recall that many operators like
$$g=f^*\g, \quad S=S^f, \quad \vol(g), \quad \nabla=\nabla^g, \quad \Delta=\Delta^g, \quad \ldots$$ 
implicitly depend on the immersion $f$. In this section their derivative 
with respect to $f$ which is called their \emph{first variation} will be calculated . 
These formulas will be used to calculate the metric gradients that are needed 
for the geodesic equation. 

This section is based on \cite{Michor118}, see also \cite{Harms2010}. 
Some but not all of the formulas were known before \cite{Besse2008, Michor102}.
More variational formulas can be found in \cite{Besse2008,Verpoort2008,Bauer2010}.

\subsection{Paths of immersions}\label{va:pa}

All of the differential-geometric concepts introduced in section \ref{no}
can be recast for a path of immersions instead of a fixed immersion. 
This allows to study variations of immersions. 
So let $f:\R \to \on{Imm}(M,N)$ be a path of immersions. By convenient calculus
\cite{MichorG}, $f$ can equivalently be seen as $f:\R \x M \to N$ 
such that $f(t,\cdot)$ is an immersion for each $t$. 
The bundles over $M$ can be replaced by bundles over $\R \x M$:
\begin{equation*}\xymatrix{
\on{pr}_2^* T^r_s M \ar[d] & 
\on{pr}_2^* T^r_s M \otimes f^*TN \ar[d] &
\Nor(f) \ar[d]\\
\R \x M & \R \x M & \R \x M
}\end{equation*}
Here $\on{pr}_2$ denotes the projection $\on{pr}_2:\R \x M \to M$.
The covariant derivative $\nabla_Z h$ is now defined for vector fields $Z$ on $\R \x M$ 
and sections $h$ of the above bundles. 
The vector fields $(\p_t, 0_M)$ and $(0_{\R}, X)$, where $X$ is a vector field on $M$, are of
special importance. 
In later sections they will be identified with $\p_t$ and $X$ 
whenever this does not pose any problems. 
Let
$$\on{ins}_t : M \to \R \x M, \qquad x \mapsto (t,x) .$$
Then by property~\ref{no:co:prop5} from section~\ref{no:co} one has for vector fields $X,Y$ on $M$
\begin{align*}
\nabla_X Tf(t,\cdot).Y &= \nabla_X T(f \o \on{ins}_t) \o Y = \nabla_X Tf \o T\on{ins}_t \o Y
\\&= \nabla_X Tf \o (0_\R,Y) \o \on{ins}_t 
= \nabla_{T\on{ins}_t \o X} Tf \o (0_\R,Y)\\&
 = \big(\nabla_{(0_\R,X)} Tf \o (0_\R,Y)\big) \o \on{ins}_t .
\end{align*}
This shows that one can recover the static situation at $t$ by using vector fields on $\R \x M$ 
with vanishing $\R$-component and evaluating at $t$.

\subsection{Directional derivatives of functions}

The following ways to denote directional derivatives of functions will be used, in particular in 
infinite dimensions.
Given a function $F(x,y)$ for instance,
$$ D_{(x,h)}F \text{ will be written as a shorthand for } \partial_t|_0 F(x+th,y).$$
Here $(x,h)$ in the subscript denotes the tangent vector with foot point $x$ and direction $h$. 
If $F$ takes values in some linear space, this linear space and its tangent space will be identified. 

\subsection{Setting for first variations}\label{va:se}

In all of this chapter, let $f$ be an immersion and $f_t \in T_f\Imm$ a tangent vector to $f$. 
The reason for calling the tangent vector $f_t$ is that in calculations  
it will often be the derivative of a curve of immersions through $f$. 
Using the same symbol $f$ for the fixed immersion
and for the path of immersions through it, one has in fact that
$$D_{(f,f_t)} F = \p_t F(f(t)).$$

\subsection{Variation of equivariant tensor fields}\label{va:ta}

Let the mapping $$F:\Imm(M,N) \to \Gamma(T^r_s M)$$ 
take values in some space of tensor fields over $M$, 
or more generally in any natural bundle over $M$, see \cite{MichorF}.

\begin{lem*}
If $F$ is equivariant 
with respect to pullbacks by 
diffeomorphisms of $M$, i.e. 
$$F(f)=(\ph^* F)(f)=\ph^* \Big(F\big((\ph\i)^*f\big)\Big) $$ 
for all $\ph \in \on{Diff}(M)$ and $f \in \Imm(M,N)$,
then the tangential variation of $F$ is its Lie-derivative:
\begin{align*}
D_{(f,Tf.f_t^\top)} F&=
\p_t|_0 F\Big(f \o Fl^{f_t^\top}_t\Big)=
\p_t|_0 F\Big((Fl^{f_t^\top}_t)^* f\Big)\\&=
\p_t|_0 \Big(Fl_t^{f_t^\top}\Big)^* \big(F(f)\big) = \L_{f_t^\top}\big(F(f)\big).
\end{align*}
\end{lem*}

This allows us to calculate the tangential variation of the pullback metric and 
the volume density, for example.

\subsection{Variation of the metric}\label{va:me}

\begin{lem*}
The differential of the pullback metric
\begin{equation*}\left\{ \begin{array}{ccl}
\Imm &\to &\Gamma(S^2_{>0} T^*M),\\
f &\mapsto &g=f^*\g
\end{array}\right.\end{equation*}
is given by
\begin{align*}
D_{(f,f_t)} g&= 2\on{Sym}\g(\nabla f_t,Tf) = -2 \g(f_t^\bot,S)+2 \on{Sym} \nabla (f_t^\top)^\flat 
\\& = -2 \g(f_t^\bot,S)+ \L_{f_t^\top} g.
\end{align*}
\end{lem*}
Here $\on{Sym}$ denotes the symmetric part of the tensor field $C$ of type $\left(\begin{smallmatrix}0\\2\end{smallmatrix}\right)$  given by
$$\big(\on{Sym}(C)\big)(X,Y):=\frac12\big(C(X,Y)+C(Y,X)\big).$$

\begin{proof}
Let $f:\R \x M \to N$ be a path of immersions. Swapping covariant derivatives as in 
section~\ref{no:sw}, formula \eqref{no:sw:to2} one gets
\begin{align*}
\p_t\big(g(X,Y)\big) &= \p_t\big( \g( Tf.X,Tf.Y ) \big)
= \g( \nabla_{\p_t}Tf.X,Tf.Y ) + \g( Tf.X, \nabla_{\p_t}Tf.Y )\\
&=\g( \nabla_X f_t,Tf.Y ) + \g( Tf.X, \nabla_Y f_t ) = \big(2 \on{Sym}\g(\nabla f_t,Tf)\big)(X,Y).
\end{align*}
Splitting $f_t$ into its normal and tangential part yields
\begin{align*}
2 \on{Sym}\g(\nabla f_t,Tf) &=
2 \on{Sym}\g(\nabla f_t^\bot + \nabla Tf.f_t^\top,Tf) \\&=
-2 \on{Sym}\g(f_t^\bot,\nabla Tf)+2 \on{Sym} g(\nabla f_t^\top,\cdot) \\&=
-2 \g(f_t^\bot,S)+2 \on{Sym} \nabla (f_t^\top)^\flat .
\end{align*}
Finally the relation
$$D_{(f,Tf.f_t^\top)} g = 2 \on{Sym} \nabla (f_t^\top)^\flat = \L_{f_t^\top} g $$
follows either from the equivariance of $g$ 
with respect to pullbacks by diffeomorphisms (see section~\ref{va:ta}) or directly from
\begin{align*}
(\L_Xg)(Y,Z)&=
\L_X\big(g(Y,Z)\big)-g(\L_XY,Z)-g(Y,\L_XZ)\\&=
\nabla_X\big(g(Y,Z)\big)-g(\nabla_XY-\nabla_YX,Z)-g(Y,\nabla_XZ-\nabla_ZX)\\&=
g(\nabla_YX,Z)+g(Y,\nabla_ZX)=
(\nabla_YX)^\flat(Z)+(\nabla_ZX)^\flat(Y)\\&=
(\nabla_YX^\flat)(Z)+(\nabla_ZX^\flat)(Y)=2 \on{Sym} \big(\nabla(X^\flat)\big)(Y,Z).\qedhere
\end{align*}
\end{proof}

\subsection{Variation of the inverse of the metric}\label{va:in}

\begin{lem*}
The differential of the inverse of the pullback metric
\begin{equation*}\left\{ \begin{array}{ccl}
\Imm &\to &\Ga\big(L(T^*M,TM)\big),\\
f &\mapsto &g\i=(f^*\g)\i
\end{array}\right.\end{equation*}
is given by
\begin{align*}
D_{(f,f_t)} g\i = D_{(f,f_t)} (f^*\g)\i =2 \g(f_t^\bot, g\i S  g\i) + \mathcal L_{f_t^\top}(g\i)
\end{align*}
\end{lem*}
\begin{proof}
\begin{align*}
\p_t g\i &= - g\i (\p_t g ) g\i
 = -g\i \big(-2 \g(f_t^\bot,S)+ \L_{f_t^\top} g\big) g\i \\
& = 2 g\i \g(f_t^\bot,S) g\i -g\i (\L_{f_t^\top} g) g\i
= 2 \g(f_t^\bot,g\i S g\i)+ \L_{f_t^\top} (g\i) \qedhere
\end{align*}
\end{proof}

\subsection{Variation of the volume density}\label{va:vo}

\begin{lem*}
The differential of the volume density
\begin{equation*}
\left\{ \begin{array}{ccl}
\Imm &\to &\Vol(M),\\
f &\mapsto &\vol(g)=\vol(f^*\g)
\end{array}\right.\end{equation*}
is given by
\begin{equation*}
D_{(f,f_t)} \vol(g) = 
\Tr^g\big(\g(\nabla f_t,Tf)\big) \vol(g)=
\Big(\on{div}^{g}(f_t^{\top})-\g\big(f_t^{\bot},\Tr^g(S)\big)\Big) \vol(g).
\end{equation*}
\end{lem*}

\begin{proof}
Let $g(t) \in \Ga(S^2_{>0}T^*M)$ be any curve of Riemannian metrics. Then
$$\p_t \vol(g)=\frac{1}{2}\on{Tr}(g\i.\p_t g)\vol(g).$$
This follows from the formula for $\vol(g)$ in a local oriented chart
$(u^1,\ldots u^m)$ on $M$:
\begin{align*}
\p_t\vol(g)&=\p_t \sqrt{\det( (g_{ij})_{ij})}\ du^1\wedge\cdots\wedge du^{m}\\
&=\frac{1}{2\sqrt{\det ((g_{ij})_{ij})}}\on{Tr}(\on{adj}(g) \p_t g)\
du^1\wedge\cdots\wedge du^{m}\\
&=\frac{1}{2\sqrt{\det ((g_{ij})_{ij})}}\on{Tr}(\det((g_{ij})_{ij})g^{-1}\p_t g)\
du^1\wedge\cdots\wedge du^{m}\\
&=\frac{1}{2}\on{Tr}(g\i.\p_t g)\vol(g)
\end{align*}
Now one can set $g = f^*\g$ and plug in the formula
$$\p_t g=\p_t (f^*\g)=2\on{Sym}\g(\nabla f_t,Tf)$$
from \ref{va:me}. 
This immediately proves the first formula:
\begin{align*}
\p_t \vol(g)&=\frac12 \Tr\big(g\i.2\on{Sym}\g(\nabla f_t,Tf) \big) 
=\Tr^g\big(\g(\nabla f_t,Tf) \big).
\end{align*}
Expanding this further yields the second formula:
\begin{align*}
\p_t \vol(g)&=\Tr^g\Big(\nabla\g( f_t,Tf)-\g( f_t,\nabla Tf) \Big)\\&
=\Tr^g\Big(\nabla\g( f_t,Tf)-\g( f_t,S) \Big)=-\nabla^*\g( f_t,Tf)-\g\big(f_t,\Tr^g(S)\big)\\&
=-\nabla^*\big((f_t^{\top})^{\flat}\big)-\g\big(f_t^{\bot},\Tr^g(S)\big)
=\on{div}(f_t^{\top})-\g\big(f_t^{\bot},\Tr^g(S)\big). 
\end{align*}
Here it has been used that
$$\nabla Tf = S \quad \text{and} \quad 
\on{div}(f_t^\top) = \Tr(\nabla f_t^\top)= \Tr^g\big((\nabla f_t^\top)^\flat\big)
= -\nabla^*\big((f_t^\top)^\flat\big).$$
Note that by \ref{va:ta}, the formula for the tangential variation 
would have followed also from the equivariance of the volume form with respect to pullbacks by 
diffeomorphisms. 
\end{proof}

\subsection{Variation of the covariant derivative}\label{va:co}

In this section, let $\nabla=\nabla^g=\nabla^{f^*\g}$ be the 
Levi-Civita covariant derivative acting on vector fields on $M$. 
Since any two covariant derivatives on $M$ differ by a tensor field, 
the first variation of $\nabla^{f^*\g}$ is tensorial. It is given by the 
tensor field $D_{(f,f_t)} \nabla^{f^*\g} \in \Ga(T^1_2 M)$.

\begin{lem*}
The tensor field $D_{(f,f_t)}\nabla^{f^*\g}$ is determined by the following relation
holding for vector fields $X,Y,Z$ on $M$:
\begin{multline*}
g\big((D_{(f,f_t)} \nabla)(X, Y),Z\big) = 
\frac12 (\nabla D_{(f,f_t)} g)\big( X \otimes Y \otimes Z
+ Y \otimes X \otimes Z	- Z \otimes X \otimes Y \big)
\end{multline*}
\end{lem*}

\begin{proof}
The defining formula for the covariant derivative is
\begin{align*}
g(\nabla_X Y,Z)&= \frac12 \Big[ Xg(Y,Z)+Yg(Z,X)-Zg(X,Y)\\&\qquad
-g(X,[Y,Z])+g(Y,[Z,X])+g(Z,[X,Y]) \Big].
\end{align*}
Taking the derivative $D_{(f,f_t)}$ yields
\begin{multline*}
(D_{(f,f_t)}g)(\nabla_X Y,Z)+g\big((D_{(f,f_t)}\nabla)(X, Y),Z\big)\\ 
\begin{aligned}
=\frac12 \Big[ & X\big((D_{(f,f_t)}g)(Y,Z)\big)+Y\big((D_{(f,f_t)}g)(Z,X)\big)-Z\big((D_{(f,f_t)}g)(X,Y)\big)\\&
-(D_{(f,f_t)}g)(X,[Y,Z])+(D_{(f,f_t)}g)(Y,[Z,X])+(D_{(f,f_t)}g)(Z,[X,Y]) \Big].
\end{aligned}
\end{multline*}
Then the result follows by replacing all Lie brackets in the above formula by covariant derivatives using 
$[X,Y]=\nabla_X Y - \nabla_Y X$
and by expanding all terms of the form $X\big((D_{(f,f_t}g)(Y,Z)\big)$ using
\begin{align*}
&X\big((D_{(f,f_t)}g)(Y,Z)\big)=\\&\qquad\qquad
(\nabla_X D_{(f,f_t)}g)(Y,Z)
+(D_{(f,f_t)}g)(\nabla_X Y,Z)
+(D_{(f,f_t)}g)(Y,\nabla_X Z).
\qedhere\end{align*}
\end{proof}

\subsection{Variation of the Laplacian}\label{va:la}

The Laplacian as defined in section \ref{no:la} 
can be seen as a smooth section of the bundle $L(T\Imm;T\Imm)$ over $\Imm$ since 
for every $f \in \Imm$ it is a mapping 
$$\De^{f^*\g}:T_f\Imm \to T_f\Imm.$$
The right way to define a first variation is to 
use the covariant derivative defined in section~\ref{sh:cov}.

\begin{lem*}
For $\De \in \Ga\big(L(T\Imm;T\Imm)\big)$, $f \in \Imm$ and $f_t,h \in T_f\Imm$ one has
\begin{align*}
(\nabla_{f_t} \Delta)(h) &=
\on{Tr}\big(g\i.(D_{(f,f_t)}g).g\i \nabla^2 h\big) 
-\nabla_{\big(\nabla^*(D_{(f,f_t)} g)+\frac12 d\on{Tr}^g(D_{(f,f_t)}g)\big)^\sharp}h \\&\qquad
+\nabla^*\big(R^{\g}(f_t,Tf)h\big)
-\Tr^g\Big( R^{\g}(f_t,Tf)\nabla h \Big).
\end{align*}
\end{lem*}

\begin{proof}
Let $f$ be a curve of immersions and $h$ a vector field along $f$. One has
$$\De : \Imm \to L(T\Imm;T\Imm), \quad \De \o f = \De^{f^*\g} : \R \to \Imm \to L(T\Imm;T\Imm). $$
Using property~\ref{no:co}.5 one gets
\begin{align*}
(\nabla_{f_t} \De)(h) &=
\big(\nabla_{\p_t} (\De \o f)\big)(h)=
\nabla_{\p_t} \De h - \De \nabla_{\p_t} h\\&=
-\nabla_{\p_t} \Tr^g(\nabla^2 h) - \De \nabla_{\p_t} h \\&=
\Tr\big(g\i (D_{(f,f_t)} g) g\i \nabla^2 h\big)- \Tr^g(\nabla_{\p_t} \nabla^2 h) 
- \De(\nabla_{\p_t} h).
\end{align*}
The term $\Tr^g(\nabla_{\p_t} \nabla^2 h)$ will be treated further. 
Let $X,Y$ be vector fields on $M$ that are constant in time. 
When they are seen as vector fields on $\R \x M$ then $\nabla_{\p_t}X=\nabla_{\p_t}Y=0$.
Using the formulas from section~\ref{no:sw} to swap covariant derivatives one gets
\begin{align*}
&(\nabla_{\p_t}\nabla^2 h)(X,Y)=
\nabla_{\p_t}(\nabla_X\nabla_Y h-\nabla_{\nabla_X Y}h)
\\&\qquad=
\nabla_X\nabla_{\p_t}\nabla_Y h+R^{\g}(f_t,Tf.X)\nabla_Y h-\nabla_{\p_t}\nabla_{\nabla_X Y}h
\\&\qquad=
\nabla_X\nabla_Y\nabla_{\p_t} h+\nabla_X\big(R^{\g}(f_t,Tf.Y)h\big)
+R^{\g}(f_t,Tf.X)\nabla_Y h\\&\qquad\qquad
-\nabla_{\nabla_X Y}\nabla_{\p_t}h-\nabla_{[\p_t,\nabla_X Y]}h
-R^{\g}(f_t,Tf.\nabla_X Y)h.
\end{align*}
The Lie bracket is
\begin{align*}
[\p_t,\nabla^{f^*\g}_X Y] = (D_{(f,f_t)}\nabla)(X,Y)
\end{align*}
since (now without the slight abuse of notation)
\begin{align*}
[(\p_t,0_M),(0_\R,\nabla^{f^*\g}_X Y)]
&=\p_s|_0\ TFl_{-s}^{(\p_t,0_M)} \o \nabla_X Y \o Fl_s^{(\p_t,0_M)} 
\\&= \big(0_\R,(D_{(f,f_t)}\nabla)(X,Y)\big).
\end{align*}
Therefore
\begin{align*}
&(\nabla_{\p_t}\nabla^2 h)(X,Y)=\\&\qquad=
(\nabla^2\nabla_{\p_t} h)(X,Y)+\nabla_X\big(R^{\g}(f_t,Tf.Y)h\big)
+R^{\g}(f_t,Tf.X)\nabla_Y h\\&\qquad\qquad
-\nabla_{(D_{(f,f_t)}\nabla)(X,Y)}h-R^{\g}(f_t,Tf.\nabla_X Y)h
\\&\qquad=
(\nabla^2\nabla_{\p_t} h)(X,Y)
+(\nabla_{Tf.X} R^{\g})(f_t,Tf.Y)h
+R^{\g}(\nabla_X f_t,Tf.Y)h\\&\qquad\qquad
+R^{\g}(f_t,\nabla_X Tf.Y)h
+R^{\g}(f_t,Tf.Y)\nabla_X h
+R^{\g}(f_t,Tf.X)\nabla_Y h\\&\qquad\qquad
-\nabla_{(D_{(f,f_t)}\nabla)(X,Y)}h-R^{\g}(f_t,Tf.\nabla_X Y)h
\\&\qquad=
(\nabla^2\nabla_{\p_t} h)(X,Y)
+(\nabla_{Tf.X} R^{\g})(f_t,Tf.Y)h
+R^{\g}(\nabla_X f_t,Tf.Y)h\\&\qquad\qquad
+R^{\g}\big(f_t,(\nabla Tf)(X,Y)\big)h
+R^{\g}(f_t,Tf.Y)\nabla_X h
+R^{\g}(f_t,Tf.X)\nabla_Y h\\&\qquad\qquad
-\nabla_{(D_{(f,f_t)}\nabla)(X,Y)}h
\\&\qquad=
(\nabla^2\nabla_{\p_t} h)(X,Y)
+ \nabla_X\big(R^{\g}(f_t,Tf.Y)h\big)
+R^{\g}(f_t,Tf.X)\nabla_Y h \\&\qquad\qquad
-\nabla_{(D_{(f,f_t)}\nabla)(X,Y)}h
\end{align*}
Putting together all terms one obtains
\begin{align*}
(\nabla_{f_t} \De)(h) &=
\Tr\big(g\i (D_{(f,f_t)} g) g\i \nabla^2 h\big)
-\Tr^g\Big(  \nabla\big(R^{\g}(f_t,Tf)h\big) \Big)\\&\qquad
-\Tr^g\Big( R^{\g}(f_t,Tf)\nabla h \Big)
+\nabla_{\Tr^g(D_{(f,f_t)}\nabla)}h\\&=
\Tr\big(g\i (D_{(f,f_t)} g) g\i \nabla^2 h\big)
+\nabla^*\big(R^{\g}(f_t,Tf)h\big) \\&\qquad
-\Tr^g\Big( R^{\g}(f_t,Tf)\nabla h \Big)
+\nabla_{\Tr^g(D_{(f,f_t)}\nabla)}h.
\end{align*}
It remains to calculate $\Tr^g(D_{(f,f_t)}\nabla)$. 
Using the variational formula for $\nabla$ from section~\ref{va:co}
one gets for any vector field $Z$ and a $g$-orthonormal frame $s_i$
\begin{align*}
&g\big(\Tr^g(D_{(f,f_t)}\nabla),Z\big) \\&\qquad=
\frac12 \sum_i (\nabla D_{(f,f_t)} g)\big( s_i \otimes s_i \otimes Z
+ s_i \otimes s_i \otimes Z	- Z \otimes s_i \otimes s_i \big) \\&\qquad=
-\big(\nabla^*(D_{(f,f_t)}g)\big)(Z) - \frac12 \Tr^g(\nabla_Z D_{(f,f_t)}g) \\&\qquad=
-\big(\nabla^*(D_{(f,f_t)}g)\big)(Z) - \frac12\nabla_Z \Tr^g(D_{(f,f_t)}g) \\&\qquad=
-\Big(\nabla^*(D_{(f,f_t)}g) + \frac12 d \Tr^g(D_{(f,f_t)}g)\Big)(Z) \\&\qquad=
-g\Big(\big(\nabla^*(D_{(f,f_t)}g) + \frac12 d \Tr^g(D_{(f,f_t)}g)\big)^\sharp,Z\Big) .
\end{align*}
Therefore 
\begin{equation*}\label{va:la:eq2}
\Tr^g(D_{(f,f_t)}\nabla) = -\big(\nabla^*(D_{(f,f_t)}g) + \frac12 d \Tr^g(D_{(f,f_t)}g)\big)^\sharp.
\qedhere
\end{equation*}
\end{proof}

\section{Sobolev-type metrics}\label{so}

\begin{ass*}
Let $P$ be a smooth section of the bundle $L(T\Imm;T\Imm)$ over $\Imm$
such that at every $f \in \Imm$ the operator 
$$P_f:T_f\Imm \to T_f\Imm$$ 
is an elliptic pseudo differential operator that is symmetric and positive with respect to 
the $H^0$-metric on $\Imm$, 
$$H^0_f(h,k) = \int_M \g(h,k)\vol(g).$$
\end{ass*}
Note that an elliptic symmetric operator is self-adjoint by \cite[26.2]{Shubin1987}.
Then $P$  induces a  metric on the set of immersions, namely
$$G^P_f(h,k)=\int_M \g(P_fh,k) \vol(g) \quad \text{for} \quad f \in \Imm, \quad h,k \in T_f\Imm.$$ 
The metric $G^P$ is positive definite since $P$ is assumed to be positive with respect to the 
$H^0$-metric.                                           
In this section, the geodesic equation on $\Imm$ and $B_i$ for the $G^P$-metric will be calculated
in terms of the operator $P$ and it will be proven that it is well-posed under some assumptions.

\subsection{Invariance of $P$ under reparametrizations}\label{so:in}

\begin{ass*}
It will be assumed that $P$ is invariant under 
the action of the reparametrization group $\Diff(M)$ acting on $\Imm(M,N)$, i.e. 
$$P=(r^{\ph})^* P \qquad \text{for all } \ph \in \Diff(M).$$ 
\end{ass*}

For any $f \in \Imm$ and $\ph \in \Diff(M)$ this means
$$P_f = (T_fr^{\ph})\i \o P_{f \o \ph} \o T_fr^{\ph}.$$
Applied to $h \in T_f\Imm$ this means 
$$P_f(h) \o \ph = P_{f \o \ph}(h \o \ph).$$

The invariance of $P$ implies that the induced metric $G^P$ is invariant 
under the action of $\Diff(M)$, too. Therefore it induces a 
unique metric on $B_i$ as explained in section~\ref{sh:rish}

\subsection{The adjoint of $\nabla P$}\label{so:ad}

The following construction is needed to express the metric gradient $H$ which is part of
the geodesic equation. $H_f$ arises from the metric $G_f$ by differentiating it with 
respect to its foot point $f \in \Imm$.
Since $G$ is defined via the operator $P$, one also needs to differentiate $P_f$ with respect to its 
foot point. As for the metric, this is accomplished by the covariant derivate.
For $P \in \Ga\big(L(T\Imm;T\Imm)\big)$ and $m \in T\Imm$ one has
$$\nabla_m P \in \Ga\big(L(T\Imm;T\Imm)\big), \qquad \nabla P \in \Ga\big(L(T^2\Imm;T\Imm)\big).$$
See section~\ref{sh:cov} for more details.

\begin{ass*}
It is assumed that there exists a smooth \emph{adjoint} 
$$\adj{\nabla P} \in \Ga\big(L^2(T\Imm;T\Imm)\big)$$ 
of $\nabla P$ in the following sense: 
\begin{equation*}
\int_M \g\big((\nabla_m P)h,k\big) \vol(g)=\int_M \g\big(m,\adj{\nabla P}(h,k)\big) \vol(g).
\end{equation*}
\end{ass*}

The existence of the adjoint needs to be checked in each specific example, 
usually by partial integration. 
For the operator $P=1+A\Delta^p$, the existence of the adjoint will be proven
and explicit formulas will be calculated in sections~\ref{la:ad} and \ref{la:ge}. 

\begin{lem*}
If the adjoint of $\nabla P$ exists, then its tangential part is determined 
by the invariance of $P$ with respect to reparametrizations:
\begin{align*}
\adj{\nabla P}(h,k)^\top &=\big(\g(\nabla Ph,k)-\g(\nabla h,Pk)\big)^\sharp \\
&=\grad^g \g(Ph,k)-\big(\g(Ph,\nabla k)+\g(\nabla h,Pk)\big)^\sharp
\end{align*}  
for $f \in \Imm, h,k \in T_f\Imm$. 
\end{lem*}
\begin{proof}
Let $X$ be a vector field on $M$. Then
\begin{align*}
(\nabla_{Tf.X} P)(h) &= 
(\nabla_{\p_t|_0} P_{f\o Fl_t^X})(h \o Fl_0^X) \\&= 
\nabla_{\p_t|_0}\big(P_{f\o Fl_t^X}(h \o Fl_t^X)\big) -
P_{f\o Fl_0^X} \big( \nabla_{\p_t|_0}(h \o Fl_t^X)\big) \\&=
\nabla_{\p_t|_0}\big(P_f(h) \o Fl_t^X\big) -
P_f \big( \nabla_{\p_t|_0}(h \o Fl_t^X)\big) \\&=
\nabla_X\big(P_f(h)) - P_f \big( \nabla_X h\big) 
\end{align*}
Therefore one has for $m,h,k \in T_f\Imm$ that
\begin{align*} &
\int_M g\big(m^\top,\adj{\nabla P}(h,k)^\top \big) \vol(g) =
\int_M \g\big(Tf.m^\top,\adj{\nabla P}(h,k)\big) \vol(g) \\&\qquad=
\int_M \g\big((\nabla_{Tf.m^\top} P)h,k\big) \vol(g)  \\&\qquad=
\int_M \g\big(\nabla_{m^\top}(Ph)-P(\nabla_{m^\top}h),k\big) \vol(g) \\&\qquad=
\int_M \big(\g(\nabla_{m^\top}Ph,k)-\g(\nabla_{m^\top}h,Pk)\big) \vol(g) \\&\qquad=
\int_M g\Big(m^\top,\big(\g(\nabla Ph,k)-\g(\nabla h,Pk)\big)^\sharp\Big) \vol(g).
\qedhere
\end{align*} 
\end{proof}

\subsection{Metric gradients}\label{so:me}

As explained in section~\ref{sh:ge}, the geodesic equation can be expressed
in terms of the metric gradients $H$ and $K$. 
These gradients will be computed now. We shall use that $P_f$ is invertible on the space of smooth 
sections. This follows because $P_f$ is an elliptic, self-adjoint, and positive operator, see the 
beginning of the proof of theorem \ref{so:we} for a detailed argument. 

\begin{lem*}
If $\adj{\nabla P}$ exists, then also $H$ and $K$ exist and are given by
\begin{align*}
K_f(h,m)&=P_f\i\Big((\nabla_m P)h+\Tr^g\big(\g(\nabla m,Tf)\big).Ph\Big) \\
H_f(h,k)&=P_f\i\Big(\adj{\nabla P}(h,k)^\bot-Tf.\big(\g(Ph,\nabla k)+\g(\nabla h,Pk)\big)^\sharp
\\&\qquad -\g(Ph,k).\Tr^g(S)\Big).
\end{align*}
\end{lem*}

\begin{proof}
For vector fields $m,h,k$ on $\Imm$ one has
\begin{equation}\label{so:me:na}
\begin{aligned}
&(\nabla_m G^P)(h,k)=
D_{(f,m)} \int_M \g(Ph,k) \vol(g) \\&\qquad\qquad
- \int_M \g\big(P(\nabla_m h),k\big) \vol(g)
- \int_M \g(Ph,\nabla_m k) \vol(g)
\\&\qquad=
\int_M D_{(f,m)}\g(Ph,k) \vol(g)
+ \int_M \g(Ph,k) D_{(f,m)}\vol(g)\\&\qquad\qquad
- \int_M \g\big(P(\nabla_m h),k\big) \vol(g)
- \int_M \g(Ph,\nabla_m k) \vol(g)
\\&\qquad=
\int_M \g\big(\nabla_m(Ph),k\big) \vol(g)
+ \int_M \g(Ph,\nabla_m k) \vol(g)\\&\qquad\qquad
+ \int_M \g(Ph,k) D_{(f,m)}\vol(g)\\&\qquad\qquad
- \int_M \g\big(P(\nabla_m h),k\big) \vol(g)
- \int_M \g(Ph,\nabla_m k) \vol(g)
\\&\qquad=
\int_M \g\big((\nabla_m P)h,k\big) \vol(g)
+ \int_M \g(Ph,k) D_{(f,m)}\vol(g)
\end{aligned}
\end{equation}
One immediately gets the $K$-gradient 
by plugging in the variational formula \ref{va:vo} for the volume form:
$$K_f(h,m)=P_f\i\Big((\nabla_m P)h+\Tr^g\big(\g(\nabla m,Tf)\big).Ph\Big).$$
To calculate the $H$-gradient, one rewrites equation~\eqref{so:me:na}
using the definition of the adjoint:
\begin{equation*}
\begin{aligned}
(\nabla_m G^P)(h,k)=
\int_M \g\big(m,\adj{\nabla P}(h,k)\big) \vol(g) + \int_M \g(Ph,k) D_{(f,m)}\vol(g).
\end{aligned}
\end{equation*}
Now the second summand is treated further using again the variational formula of the volume density 
from section~\ref{va:vo}:
\begin{align*}
&\int_M \g(Ph,k) D_{(f,m)}\vol(g)=
\int_M \g(Ph,k) \Tr^g\big(\g(\nabla m,Tf)\big) \vol(g) \\
&\qquad=
\int_M \g(Ph,k) \Tr^g\big(\nabla\g(m,Tf)-\g(m,\nabla Tf)\big) \vol(g) \\
&\qquad=
\int_M \g(Ph,k) \Big(-\nabla^*\g(m,Tf)-\g\big(m,\Tr^g(S)\big)\Big) \vol(g) \\
&\qquad=
-\int_M g^0_1\big(\nabla\g(Ph,k),\g(m,Tf)\big)\vol(g)-\int_M \g(Ph,k) \g\big(m,\Tr^g(S)\big) \vol(g) \\
&\qquad=
\int_M \g\big(m,-Tf.\grad^g\g(Ph,k)-\g(Ph,k) \Tr^g(S)\big) \vol(g) 
\end{align*} 
Collecting terms one gets that 
\begin{align*}
&G_f^P(H_f(h,k),m)=(\nabla_m G^P)(h,k)\\&\qquad=
\int_M \g\big(m,\adj{\nabla P}(h,k)-Tf.\grad^g\g(Ph,k)-\g(Ph,k) \Tr^g(S)\big) \vol(g) 
\end{align*}
Thus the $H$-gradient is given by
\begin{align*}
H_f(h,k)=P\i\Big(\adj{\nabla P}(h,k)-Tf.\grad^g\g(Ph,k) -\g(Ph,k).\Tr^g(S)\Big)
\end{align*}
The highest order term $\grad^g\g(Ph,k)$ cancels out when taking into 
account the formula for the tangential part of the adjoint from section~\ref{so:ad}:
\begin{align*}
H_f(h,k)=P\i\Big(&\adj{\nabla P}(h,k)^\bot-Tf.\big(\g(Ph,\nabla k)+\g(\nabla h,Pk)\big)^\sharp
\\& -\g(Ph,k).\Tr^g(S)\Big). 
\qedhere
\end{align*}
\end{proof}

\subsection[Geodesic equation]{Geodesic equation on immersions}\label{so:ge}

The geodesic equation for a general metric on $\Imm(M,N)$ has been calculated
in section~\ref{sh:ge} and reads as 
$$\nabla_{\p_t} f_t = \frac12 H_f(f_t,f_t) - K_f(f_t,f_t). $$
Plugging in the formulas for $H,K$ derived in the last section yields the following theorem. 

\begin{thm*}
The geodesic equation for a Sobolev-type metric $G^P$ on immersions is given by
\begin{align*}
\nabla_{\p_t} f_t= &\frac12P\i\Big(\adj{\nabla P}(f_t,f_t)^\bot-2.Tf.\g(Pf_t,\nabla f_t)^\sharp
-\g(Pf_t,f_t).\Tr^g(S)\Big)
\\&
-P\i\Big((\nabla_{f_t}P)f_t+\Tr^g\big(\g(\nabla f_t,Tf)\big) Pf_t\Big).
\end{align*}
\end{thm*}

\subsection[Geodesic equation]{Geodesic equation on immersions in terms of the momentum}\label{so:gemo}

The geodesic equation in terms of the momentum has been calculated in section~\ref{sh:gemo}
for a general metric on immersions. For a Sobolev-type metric $G^P$, the momentum 
$G^P(f_t,\cdot)$ takes the form 
$$p=Pf_t\otimes\vol(g): \R \to T^*\Imm$$ 
since all other parts of the metric (namely the integral and $\g$) 
are constant and can be neglected. 

\begin{thm*}
The geodesic equation written in terms of the momentum for a Sobolev-type metric $G^P$
on $\Imm$ is given by:
\begin{equation*}
\left\{\begin{aligned}
p&=Pf_t\otimes\vol(g) 
\\
\nabla_{\p_t}p &= \frac12\big(\adj{\nabla P}(f_t,f_t)^\bot-2Tf.\g(Pf_t,\nabla f_t)^\sharp
-\g(Pf_t,f_t)\Tr^g(S)\big)\otimes\vol(g)
\end{aligned}\right.
\end{equation*}
\end{thm*}

\subsection{Well-posedness of the geodesic equation}\label{so:we}

It will be proven that the geodesic equation for a Sobolev-type metric $G^P$ on $\Imm$ is well-posed 
under some assumptions on $P$. 
These assumptions are satisfied for the operator $1+A\De^p$ considered in section~\ref{la}.
It will also be shown that $(\pi,\exp)$ is a diffeomorphism from a neighbourhood 
of the zero section in $T\Imm$ to a neighbourhood of the diagonal in $\Imm \x \Imm$. 

Before we can state the theorem, we have to introduce Sobolev completions of
the relevant spaces of mappings. 
More information can be found in \cite{Shubin1987}, \cite{EichhornFricke1998}, and in  
\cite{Eichhorn2007}. 
We consider Sobolev completions 
of $\Ga(E)$, where $E \to M$ is a vector bundle. 
First we choose a fixed (background) Riemannian metric $\hat g$ on $M$ and its
covariant derivative $\nabla^M$.
We equip $E$ with a (background) fiber Riemannian metric $\hat g^E$ 
and a compatible covariant derivative $\hat \nabla^E$.
Then the {\it Sobolev space} $H^k(E)$ is the Hilbert space completion 
of the space of smooth sections $\Ga(E)$ in the Sobolev norm
$$\|h\|_k^2 = \sum_{j=0}^k \int_M  (\hat g^E\otimes \hat g^0_j)\big((\hat \nabla^E)^j h, 
(\hat \nabla^E)^j h\big)\vol(\hat g).$$
This Sobolev space does not depend on the choices of $\hat g$, $\nabla^M$, $\hat g^E$ and 
$\hat \nabla^E$ since $M$ is compact: 
The resulting norms are equivalent. 

We shall need the following results (see \cite{Eichhorn2007}, e.g.):
\newtheorem*{SL}{Sobolev lemma}
\newtheorem*{MP}{Module property of Sobolev spaces}
\begin{SL}
If $k>\dim(M)/2$ then the identy on $\Ga(E)$ extends to a
injective bounded linear mapping $H^{k+p}(E)\to C^p(E)$ where $C^p(E)$ carries the 
supremum norm of all derivatives up to order $p$.
\end{SL}
\begin{MP}
If $k>\dim(M)/2$ then pointwise evaluation $H^k(L(E,E))\x H^k(E)\to H^k(E)$ is
bounded bilinear. Likewise all other pointwise contraction operations are multilinear bounded
operations.
\end{MP}

This allows us to define Sobolev completions of $\Imm$ and $T\Imm$.
In the canonical charts for $\Imm(M,N)$ centered at an immersion $f_0$, 
every immersion corresponds to a section of
the vector bundle $f_0^*TN$ over $M$ (see \cite[section~42]{MichorG}). 
The smooth Sobolev manifold $\Imm^k(M,N)$ (for $k>\dim(M)/2+1$) 
is constructed by gluing together the Sobolev completions
$H^k(f_0^*TN)$ of each canonical chart.
One has
$$\Imm^{k+1}(M,N) \subset \Imm^k(M,N),\qquad \bigcap_{k}\Imm^k(M,N) = \Imm(M,N).$$
Similarly, Sobolev completions of the space $T\Imm \subset C^\infty(M,TN)$ are defined
as $H^k$-mappings from $M$ into $TN$, i.e. $T\Imm^k=H^k(M,TN)$. 

\begin{ass}
$P,\nabla P$ and $\adj{\nabla P}^\bot$ are smooth sections of the bundles
\begin{equation*}\xymatrix{
L(T\Imm;T\Imm) \ar[d] & L^2(T\Imm;T\Imm) \ar[d] & L^2(T\Imm;T\Imm) \ar[d] \\
\Imm & \Imm & \Imm,
}\end{equation*}
respectively.
Viewed locally in trivializations of these bundles,
$$P_f h, \qquad (\nabla P)_f (h,k), \qquad \big(\adj{\nabla P}_f(h,k)\big)^\bot$$
are pseudo-differential operators of order $2p$ in $h,k$ separately. 
As mappings in the footpoint $f$ they are non-linear, and it is assumed that they are a 
composition of operators of the following type:
\newline
\textrm{(a)} Local operators of order $l\le 2p$, i.e., nonlinear differential operators
$$A(f)(x)=A(x,\hat \nabla^{l}f(x),\hat \nabla^{l-1}f(x),\dots,\hat \nabla f(x), f(x)),$$ 
\newline
\textrm{(b)}  Linear pseudo-differential operators of degrees $l_i$, 
\newline
such that the total (top) order of the composition is $\le 2p$.
\end{ass}

\begin{ass}
For each $f\in \Imm(M,N)$, the operator $P_f$ 
is an elliptic pseudo-differential operator of order $2p$ for $p>0$ 
which is positive and symmetric with respect to the $H^0$-metric on $\Imm$, i.e.
$$\int_M \g(P_f h,k)\vol(g) = \int_M \g(h,P_f k)\vol(g) \qquad \text{for } h,k \in T_f\Imm.$$
\end{ass}

\begin{ass} 
$P$ is invariant under the action of $\Diff(M)$. 
(See section~\ref{so:in} for the definition of invariance.)
\end{ass}

\begin{thm*}
Let $p\ge 1$ and $k>\dim(M)/2+1$, and let
$P$ satisfy assumptions 1--3.

Then the initial value problem for the geodesic equation 
\thetag{\ref{so:ge}}
has unique local solutions in the Sobolev manifold $\Imm^{k+2p}$ of
$H^{k+2p}$-immersions. The solutions depend smoothly on $t$ and on the initial
conditions $f(0,\;.\;)$ and $f_t(0,\;.\;)$.
The domain of existence (in $t$) is uniform in $k$ and thus this
also holds in $\Imm(M,N)$.

Moreover, in each Sobolev completion $\Imm^{k+2p}$, the Riemannian exponential mapping $\exp^{P}$ exists 
and is smooth on a neighborhood of the zero section in the tangent bundle, 
and $(\pi,\exp^{P})$ is a diffeomorphism from a (smaller) neigbourhood of the zero 
section to a neighborhood of the diagonal in 
$\Imm^{k+2p}\x \Imm^{k+2p}$.
All these neighborhoods are uniform in $k>\dim(M)/2+1$ and can be chosen 
$H^{k_0+2p}$-open, for $k_0 > \dim(M)/2+1$. Thus both properties of the exponential 
mapping continue to hold in $\Imm(M,N)$.
\end{thm*}

This proof is partly an adaptation of \cite[section 4.3]{Michor107}. 
It works in three steps: First, the geodesic equation 
is formulated as the flow equation of a smooth vector field on each Sobolev completion
$T\Imm^{k+2p}$. 
Thus one gets local existence and uniqueness of solutions. Second, it is shown that 
the time-interval where a solution exists does not depend on the order of the Sobolev space
of immersions. Thus one gets solutions on the intersection of all Sobolev spaces, which 
is the space of smooth immersions. Third, a general argument involving the inverse function theorem 
on Banach spaces proves the claims 
about the exponential map.

\begin{proof} By assumption~1 the mapping
$P_f h$ is of order $2p$ in $f$ and in $h$ where $f$ is the footpoint of $h$.
Therefore
$f \mapsto P_f$ extends to a smooth section of the smooth Sobolev bundle
\begin{equation*}
L\big(T\Imm^{k+2p};T\Imm^k\mid \Imm^{k+2p} \big) \to \Imm^{k+2p}, 
\end{equation*}
where $T\Imm^k\mid \Imm^{k+2p}$ denotes the space of all $H^k$ tangent vectors 
with foot point a $H^{k+2p}$ immersion, i.e., the restriction of the bundle $T\Imm^k\to \Imm^{k}$ 
to $\Imm^{k+2p}\subset\Imm^k$.

This means that $P_f$ is a bounded linear operator
$$P_f \in L\big(H^{k+2p}(f^*TN),H^k(f^*TN)\big)\quad\text{for}\quad f\in \Imm^{k+2p}.$$ 
It is injective since it is positive. 
As an elliptic operator, 
it is an unbounded operator on the Hilbert completion of $T_f\Imm$ with respect to the 
$H^0$-metric, and a Fredholm operator $H^{k+2p}\to H^k$ for each $k$. 
It is selfadjoint elliptic, thus by \cite[theorem 26.2]{Shubin1987} it has vanishing index.
Since it is injective, it is thus also surjective.

By the implicit function theorem on Banach spaces, 
$f\mapsto P_f\i$ is then a smooth section of the smooth Sobolev bundle 
\begin{equation*}
L\big(T\Imm^k \mid \Imm^{k+2p} ; T\Imm^{k+2p}\big) \to \Imm^{k+2p}
\end{equation*}
As an inverse of an elliptic pseudodifferential operator, 
$P_f\i$ is also an elliptic pseudo-differential operator of order $-2p$. 

By assumption~1 again, $(\nabla P)_f(m,h)$ and $\big(\adj{\nabla P}_f(m,h)\big)^\bot$
are of order $2p$ in $f,m,h$ (locally). 
Therefore $f\mapsto P_f$ and $f\mapsto \adj{\nabla P}^\bot$ 
extend to smooth sections of the smooth Sobolev bundle 
\begin{equation*}
L^2\big(T\Imm^{k+2p}; T\Imm^k \mid \Imm^{k+2p}\big) \to \Imm^{k+2p}
\end{equation*}

Using the module property of Sobolev spaces
and counting the order of all remaining terms in the geodesic equation 
\ref{so:ge}, one obtains that the Christoffel symbols
\begin{align*}
\frac12 H_f(h,h)-K_f(h,h) &=\frac12P\i\Big(\adj{\nabla P}(h,h)^\bot-2.Tf.\g(Ph,\nabla h)^\sharp
\\&\qquad
-\g(Ph,h).\Tr^g(S)-(\nabla_{h}P)h-\Tr^g\big(\g(\nabla h,Tf)\big) Ph\Big)
\end{align*}
extend to a smooth $(C^\infty)$ section 
of the smooth Sobolev bundle
\begin{equation*}
L^2_{\text{sym}}\big(T\Imm^{k+2p}; T\Imm^{k+2p}\big) \to \Imm^{k+2p}
\end{equation*}
Thus $h\mapsto \tfrac12 H(h,h)-K(h,h)$ is a smooth quadratic mapping $T\Imm \to T\Imm$ which 
extends to smooth quadratic mappings $T\Imm^{k+2p} \to T\Imm^{k+2p}$ for each 
$k\ge \frac{\dim(2)}2+1$.
The geodesic equation 
$$\nabla^{\g}_{\p_t}f_t = \frac12 H_f(f_t,f_t)-K_f(f_t,f_t)$$
can be reformulated using the linear connection  $C^g:TN \x_N TN \to TTN$ (horizontal lift mapping) 
of $\nabla^{\g}$,  
see \cite[section~24.2]{MichorH}:
\begin{align*}
\p_t f_t &=C\Big( \frac12 H_f(f_t,f_t)-K_f(f_t,f_t),f_t\Big).
\end{align*}
The right-hand side is a smooth vector field on $T\Imm^{k+2p}$,
the geodesic spray. Note that the restriction to $T\Imm^{k+1+2p}$ of the geodesic spray on 
$T\Imm^{k+2p}$ equals the geodesic spray there.
By the theory of smooth ODE's on Banach spaces, the flow of 
this vector field exists in $T\Imm^{k+2p}$ and is smooth in time and in the initial condition. 

Consider a $C^\infty$ initial condition $h_0 \in T\Imm$ with foot point $f_0\in \Imm$. 
Suppose the 
trajectory $\on{Fl}^k_t(h_0)$ of
geodesic spray through these initial conditions in $T\Imm^{k+2p}$ maximally exists for $t\in (-a_k,b_k)$, 
and the trajectory $\on{Fl}^{k+1}_t(h_0)$ in $T\Imm^{k+1+2p}$ 
maximally exists for $t\in(-a_{k+1},b_{k+1})$ with 
$b_{k+1}<b_k$, say. By
uniqueness of solutions one has $\on{Fl}^{k+1}_t(h_0)=\on{Fl}^{k}_t(h_0)$ for
$t\in (-a_{k+1,}b_{k+1})$.
We now write $\hat\nabla$ for the covariant derivative induced by $\g$ on $N$ and the background 
metric $\hat g$ on $M$.
Let $X$ be a vector field on $M$. 
Applying $\nabla^{\g}_X=:\hat\nabla_X$ to the geodesic equation 
and swapping covariant derivatives yields: 
\begin{align*}
\nabla^{\g}_{\p_t} \nabla^{\g}_{\p_t} Tf.X 
&= 
\nabla^{\g}_{\p_t} \nabla^{\g}_X f_t  
=
\nabla^{\g}_{X} \nabla^{\g}_{\p_t} f_t +R^{\g}(f_t,Tf.X)f_t
\\&=
\nabla^{\g}_{X} \big(\tfrac12 H_f(f_t,f_t)-K_f(f_t,f_t)\big)+R^{\g}(f_t,Tf.X)f_t 
\tag{A}
\end{align*}
Note that $i_X\hat\nabla_{\p_t}\nabla^{\g} f_t = \nabla^{\g}_{\p_t}\nabla^{\g}_X f_t - 
\nabla^{\g}_{\nabla^{\g}_{\p_t}X} f_t = \nabla^{\g}_{\p_t}\nabla^{\g}_X f_t-0$. Thus we can omit $X$ and 
rewrite \thetag{A} as an equation for $Tf$.
We aim to rewrite equation \thetag{A} as a linear first order equation for the highest derivative 
whose coefficients are given by $\on{Fl}^k_t(h_0)$ and thus exist beyond $(-a_{k+1},b_{k+1})$. For this we 
have to pass to one (ore more) canonical chart for $\Imm$ and the induced trivializations of all bundles as 
before and in assumption \thetag{1}.
Then $f$ itself has values in a vector space and we may regard \thetag{A} as a vector valued  1-form on $M$.
So we rewrite \thetag{A} as:
\begin{align*}
\hat\nabla_{\p_t} Tf &= \nabla^{\g} f_t
\\
\hat\nabla_{\p_t} \nabla^{\g} f_t  
&=
\nabla^{\g} \big(\tfrac12 H_f(f_t,f_t)-K_f(f_t,f_t)\big) + R^{\g}(f_t,Tf)f_t 
\\&=
Y^1(f,f_t)(\hat\nabla^{2p}Tf) + Y^2(f,f_t)(\hat\nabla^{2p+1}f_t) +	 Y^3(f,f_t).
\tag{B}\end{align*}
We claim that
\thetag{B} consists of: 
\newline $\bullet$
The smooth expression $Y^1(f,f_t)(\hat\nabla^{2p+1}f)$ which is \emph{linear} and of order 0 in 
$\hat\nabla^{2p+1}f$ and where $Y^1(f,f_t)$ is of order $\le 2p$ in $f,f_t$; order here means that the 
expression prolongs continuously to the corresponding Sobolev spaces. 
\newline $\bullet$
The smooth expression $Y^2_X(f,f_t)(\hat\nabla^{2p+1} f_t)$ which is \emph{linear} and of order 0 in 
$\hat\nabla^{2p+1}f_t$ and where $Y^2(f,f_t)$ is of order $\le 2p$ in $f,f_t$.
\newline $\bullet$
The smooth expression $Y^3_X(f,f_t)$ of order $\le 2p$ in $f,f_t$. 
\newline
To see this 
we claim that the highest derivatives of order $2p+1$ of $f$ and $f_t$ appear only linearly
in \thetag{A}. 
This claim follows from assumption~1: 
\newline
(a) For a local operator we can apply the chain rule: 
The highest derivative of $f$ appears only linearly.
\newline
(b) For a linear pseudo differential operator $A$ of order $k$ 
the commutator $[\hat \nabla,A]$ is a pseudo-differential operator of order $k$ again. 
\newline
On the left hand side of \thetag{B} we write $Tf=:u$ and $\hat\nabla f_t =: v$.
On the right hand side of \thetag{B} we write 
$\hat \nabla^{2p} Tf = \hat \nabla^{2p}u$ and 
$\hat \nabla^{2p+1}f_t = \hat \nabla^{2p} v$ for the highest derivatives only.
Then the system \thetag{B} becomes:
\begin{align*}
\hat\nabla_{\p_t} u &= v
\\
\hat\nabla_{\p_t} v
&=
Y^1(f,f_t)(\hat\nabla^{2p}u) + Y^2(f,f_t)(\hat\nabla^{2p}v) +	 Y^3(f,f_t).
\tag{C}\end{align*}
The coefficients $f,f_t$ in \thetag{C} exist for $t\in (-a_k,b_k)$ as $\on{Fl}^k_t(h_0)$.
Then \thetag{C} is a bounded and smooth inhomogeneous linear ODE for 
$(u,v)\in \Om^1(M,T\Imm^{k+2p})$, i.e., in a Banach space.
This equation 
therefore has a solution $(u(t),v(t))$ for all $t$ for which the
coefficients exists, thus for all $t\in (a_k,b_k)$, which is unique for the initial values $u_0 = T f_0$ and 
$v_0 = \hat\nabla h_0$. The limit 
$\lim_{t\nearrow b_{k+1}} (u(t),v(t))$ exists in $\Om^1(M,T\Imm^{k+2p})$ 
and by continuity it equals $(\hat \nabla Tf,\hat \nabla f_t)$
for $t=b_{k+1}$. Thus the flow line $\on{Fl}^k_t(h_0)$ was not
maximal and can be continued. So assuming $b_{k+1}<b_k$ leads to a contradiction, and thus 
$(-a_{k+1},b_{k+1})=(-a_k,b_k)$.  
Iterating this procedure one concludes that the flow line $\on{Fl}^{\infty}_t(h_0)$ exists in
$\bigcap_{k\ge \frac{\dim(M)}2+1} T\Imm^{k+2p}= T\Imm$. 

It remains to check the properties of the Riemannian exponential mapping $\exp^P$.
It is given by $\exp^P_{f}(h)= c(1)$ where $c(t)$ is the geodesic emanating from  
value $f$ with initial velocity $h$.
Let $k_0>\dim(M)/2+1$ and $k\ge k_0$.
On each space $T\Imm^{k+2p}$, the properties claimed follow
from local existence and uniqueness of solutions to 
the flow equation of the geodesic spray, from
the form of the geodesic equation $f_{tt}=\tfrac12 H(f_t,f_t)-K(f_t,f_t)$
when it is written down in a chart,
namely linearity in $f_{tt}$ and bilinearity in $f_t$, 
and from the inverse function theorem 
which holds on each of the Sobolev spaces $\Imm^{k+2p}$.
See for example \cite[22.6 and 22.7,]{MichorH} for a detailed proof 
which works without any change in notation.

$\Imm^{k_0+2p}$ contains $\Imm^{k+2p}$ for $k>k_0$.
Since the spray on $\Imm^{k_0+2p}$ restricts to the spray on each $\Imm^{k+2p}$,
the exponential mapping $\exp^P$ and the inverse $(\pi,\exp^P)\i$ on $\Imm^{k_0+2p}$ restrict 
to the corresponding mappings on each $\Imm^{k+2p}$. Thus the neighborhoods of existence are 
uniform in $k$ and can be chosen $H^{k_0+2p}$-open.
\end{proof}

\subsection{Momentum mappings}\label{so:mo}

Recall that by assumption, the operator $P$ is invariant under the action of the 
reparametrization group $\on{Diff}(M)$. Therefore
the induced metric $G^P$ is invariant under this group action, too.  
According to \cite[section~2.5]{Michor107} one gets:

\begin{thm*}
The reparametrization momentum, which is the momentum mapping 
corresponding to the action of $\Diff(M)$ on $\Imm(M,N)$, 
is conserved along any geodesic $f$ in $\Imm(M,N)$: 
\begin{align*}
&\forall X\in\X(M): \int_M  \g( Tf.X,Pf_t ) \vol(g) \\
\intertext{or equivalently}
&g\bigl((Pf_t )^\top\bigr) \vol(g) \in\Ga(T^*M\otimes_M\vol(M))
\end{align*}
is constant along $f$.
\end{thm*}

\subsection{Horizontal bundle}\label{so:ho}

The splitting of $T\Imm$ into horizontal and vertical subspaces will be calculated
for Sobolev-type metrics $G^P$. See section~\ref{sh:rish} for the definitions. 
By definition, a tangent vector $h$ to $f \in \Imm(M,N)$ is horizontal if
and only if it is $G^P$-perpendicular to
the $\on{Diff}(M)$-orbits. This is the case if and only if 
$\g( P_f h(x), T_x f .X_x) = 0$ at every point $x \in M$. 
Therefore the horizontal bundle at the point $f$ equals 
\begin{align*}
&\big\{h\in T_f\Imm: P_fh(x)\,\bot\, T_x f(T_xM)\text{ for all 
}x\in M\}
=\big\{h : (P_fh)^\top = 0\big\}. 
\end{align*}
Note that the horizontal bundle consists of vector fields that are normal to $f$ 
when $P=\Id$, i.e. for the $H^0$-metric on $\Imm$.

Let us work out the $G^P$-decomposition of $h$ into vertical and 
horizontal parts. This decomposition is written as
\begin{equation*}
h= Tf.h^{\text{ver}} + h^{\text{hor}}.
\end{equation*}
Then 
\begin{align*}
P_fh = P_f (Tf.h^{\text{ver}}) + P_f h^{\text{hor}} \quad \text{with} \quad
(P_fh)^\top = (P_f (Tf.h^{\text{ver}}))^\top + 0. 
\end{align*}
Thus one considers the operators
\begin{align*}
&P_f^\top :\X(M) \to \X(M), &\qquad  P_f^\top (X) &= \big(P_f(Tf.X)\big)^\top,
\\
&P_{f,\bot}:\X(M)\to \Ga\big(\Nor(f)\big) \subset C^\infty(M,TN), &  P_{f,\bot}(X) &=  
\big(P_f(Tf.X)\big)^\bot.
\end{align*}
The operator $P_f^\top $ is unbounded, positive and symmetric on the Hilbert completion
of $T_f\Imm$ with respect to the $H^0$-metric since one has 
\begin{align*}
\int_M g(P_f^\top X,Y)\vol(g) &= 
\int_M \g(Tf.P_f^\top X,Tf.Y)\vol(g) 
\\&
= \int_M \g(P_{f}(Tf.X),Tf.Y)\vol(g) 
\\&
= \int_M g(P_f^\top Y,X)\vol(g), 
\\
\int_M g(P_f^\top X,X)\vol(g) &= 
\int_M \g(P_{f}(Tf.X),Tf.X)\vol(g) \; > 0 \quad\text{  if }X\ne 0.
\end{align*}
Let $\si^{P_f}$ and $\si^{P_f^\top}$ denote the principal symbols of $P_f$ and $P_f^\top$, 
respectively. Take any $x \in M$ and $\xi\in T^*_xM\setminus\{0\}$.
Then $\si^{P_f}(\xi)$ is symmetric, positive definite on $(T_{f(x)}N,\g)$. 
This means that one has for any $h,k \in T_{f(x)}N$ that
\begin{equation*}
\g\big(\si^{P_f}(\xi)h,k\big) = \g\big(h,\si^{P_f}(\xi)k\big), \qquad
\g\big(\si^{P_f}(\xi)h,h\big) > 0 \text{ for } h \neq 0.
\end{equation*}
The principal symbols $\si^{P_f}$ and $\si^{P_f^\top}$ are related by
\begin{equation*}
g\big(\si^{P_f^\top}(\xi)X,Y\big) = \g\big(Tf.\si^{P_f^\top}(\xi)X,Tf.Y\big)
= \g\big(\si^{P_f}(\xi)Tf.X,Tf.Y\big),
\end{equation*}
where $X,Y \in T_xM$.
Thus $\si^{P_f^\top}(\xi)$ is symmetric, positive definite on $(T_xM,g)$.
Therefore $P_f^\top $ is again elliptic, thus it is selfadjoint, 
so its index (as operator $H^{k+2p}\to H^{k}$) vanishes. It 
is injective (since positive) with vanishing index (since self-adjoint elliptic, by \cite[theorem 26.2]{Shubin1987}) 
hence it is bijective and thus invertible by the open mapping theorem. 
Thus it has been proven: 

\begin{lem*}
The decomposition of $h \in T\Imm$ into its vertical and horizontal components 
is given by
\begin{align*}
h^{\text{ver}} &= (P_f^\top )\i\big((P_fh)^\top\big),
\\
h^{\text{hor}} &= h - Tf.h^{\text{ver}} =  h - Tf.(P_f^\top )\i\big((P_fh)^\top\big).
\end{align*}
\end{lem*}

\subsection{Horizontal curves}\label{so:ho2}

To establish the one-to-one correspondence between horizontal 
curves in $\Imm$ and curves in shape space that has been described in theorem~\ref{sh:sub}, 
one needs the following property: 

\begin{lem*}
For any smooth path $f$ in $\Imm(M,N)$ there exists a
smooth path $\ph$ in $\on{Diff}(M)$ with $\ph(0,\;.\;)=\on{Id}_M$ 
depending smoothly on $f$ such that
the path $\tilde f$ given by $\tilde f(t,x)=f(t,\ph(t,x))$ is horizontal:
$$\g \big( P_{\tilde f}(\p_t\tilde f),T\tilde f.TM \big) =0.$$
Thus any path in shape space can be lifted to a horizontal path of immersions.
\end{lem*}

The basic idea is to write the path $\ph$ as the integral curve of a time dependent vector field. 
This method is called the Moser-trick (see \cite[Section 2.5]{Michor102}).

\begin{demo}{Proof}
Since $P$ is invariant, one has $(r^\ph)^* P = P$ or 
$P_{f\o \ph}(u\o \ph)=(P_fu)\o\ph$ for $\ph\in\on{Diff}(M)$. 
In the following $f\circ\varphi$ will denote the map $f(t, \varphi(t,x))$, etc.
One looks for $\ph$ as the integral curve of a time dependent vector field $\xi(t,x)$ on
$M$, given by $\ph_t=\xi\o \ph$.
The following expression must vanish for all $x \in M$ and $X_x \in T_x M$:
\begin{align*}
0&=\g\Big( P_{f\o\ph}\big(\p_t(f\circ \varphi)\big)(x), T(f\circ \varphi).X_x \Big) \\&=
\g\Big( P_{f\o\ph}\big((\p_t f)\circ\varphi +Tf.(\p_t\varphi)\big)(x), T(f\circ \varphi).X_x \Big)
\\ &=
\g\Big( \big((P_{f}(\p_tf)) +P_{f}(Tf.\xi)\big)\big(\ph(x)\big),Tf\circ T\varphi.X_x \Big) 
\end{align*}
Since $T\varphi$ is surjective, $T\varphi.X$ exhausts the tangent space $T_{\varphi(x)}M$, and one has
$$\big((P_{f}(\p_tf)) +P_{f}(Tf.\xi)\big)\big(\ph(x)\big)\quad \perp \quad f.$$
This holds for all $x \in M$, and by the surjectivity of $\varphi$, one also has that
$$\big((P_{f}(\p_tf)) +P_{f}(Tf.\xi)\big)(x)\quad \perp \quad f$$
at all $x \in M$. 
This means that the tangential part $\big(P_{f}(\p_tf) +  P_f(Tf.\xi)\big)^\top$ vanishes.
Using the time dependent vector field
$$\xi=-(P_f^\top )\i \big((P_f \p_t f)^\top\big)$$
and its flow $\ph$ achieves this.
\qed\end{demo}

\subsection[Geodesic equation]{Geodesic equation on shape space}\label{so:gesh}

By the previous section and theorem~\ref{sh:sub}, 
geodesics in $B_i$ correspond exactly to horizontal geodesics in $\Imm$. 
The equations for horizontal geodesics in the space of immersions have been written down in 
section~\ref{sh:gesh}. Here they are specialized to Sobolev-type metrics: 
\begin{thm*} 
The geodesic equation on shape space for a Sobolev-type metric $G^P$ 
is equivalent to the set of equations
\begin{equation*}\left\{\begin{aligned}
f_t &= f_t^\hor \in \Hor, \\
(\nabla_{\p_t} f_t)^\hor &= 
\frac12P\i\Big(\adj{\nabla P}(f_t,f_t)^\bot -\g(Pf_t,f_t).\Tr^g(S)\Big)\\
&\qquad-P\i\Big(\big((\nabla_{f_t}P)f_t\big)^\bot-\Tr^g\big(\g(\nabla f_t,Tf)\big) Pf_t\Big),
\end{aligned}\right.\end{equation*}
where $f$ is a horizontal path of immersions.
\end{thm*}
These equations are not handable very well since taking the horizontal part
of a vector to $\Imm$ involves inverting an elliptic pseudo-differential operator, 
see section~\ref{so:ho}. 
However, the formulation in the next section is much better.

\subsection[Geodesic equation]{Geodesic equation on shape space in terms of the momentum}\label{so:geshmo}

The geodesic equation in terms of the momentum has been derived in section~\ref{sh:geshmo} 
for a general metric on shape space. Now it is specialized to Sobolev-type metrics
using the formula for the $H$-gradient from section~\ref{so:me}. 

As in section~\ref{so:gemo} the momentum $G^P(f_t,\cdot)$ is identified with
$Pf_t \otimes \vol(g)$. By definition, the momentum is horizontal if it annihilates all 
vertical vectors. This  is the case if and only if $Pf_t$ is normal to $f$. 
Thus the splitting of the momentum in horizontal and vertical parts is given by
$$Pf_t \otimes \vol(g) = (Pf_t)^\bot \otimes \vol(g) + Tf.(Pf_t)^\top \otimes \vol(g).$$
This is much simpler than the splitting of the velocity in horizontal and vertical parts where 
a pseudo-differential operator has to be inverted, see section~\ref{so:ho}.
Thus the following version of the geodesic equation on shape space is the easiest to solve.

\begin{thm*}
The geodesic equation on shape space is equivalent to the set of equations
for a path of immersions $f$:
\begin{equation*}
\left\{\begin{aligned}
p &= Pf_t \otimes \vol(g), \qquad Pf_t = (Pf_t)^\bot, \\
(\nabla_{\p_t}p)^\hor &= \frac12 \Big(\adj{\nabla P}(f_t,f_t)^\bot-\g(Pf_t,f_t).\Tr^g(S)\Big) \otimes \vol(g).
\end{aligned}\right.
\end{equation*}
\end{thm*}

The equation for geodesics on $\Imm$ without the horizontality condition is
$$\nabla_{\p_t}p = \frac12\big(\adj{\nabla P}(f_t,f_t)^\bot-2Tf.\g(Pf_t,\nabla f_t)^\sharp
-\g(Pf_t,f_t).\Tr^g(S)\big)\otimes\vol(g),
$$
see section~\ref{so:gemo}. 
It has been proven in section~\ref{sh:gesh} that the vertical part of this equation 
is satisfied automatically when the geodesic is horizontal. 
Nevertheless this will be checked by hand because the proof is much simpler here
than in the general case. 

If $f_t$ is horizontal then by definition $Pf_t$ is normal to $f$. Thus one has for any $X \in \X(M)$ that
\begin{align*}
g\big((\nabla_{\p_t}Pf_t)^\top, X \big) &= 
\g(\nabla_{\p_t}Pf_t, Tf.X ) = 
0-\g(Pf_t,\nabla_{\p_t}Tf.X)\\&=
-\g(Pf_t,\nabla_X f_t)=
-g\big(\g(Pf_t,\nabla f_t)^\sharp,X).
\end{align*}
Thus 
\begin{align*}
\big(\nabla_{\p_t} p\big)^\vert &=
\big((\nabla_{\p_t} Pf_t) \otimes \vol(g) + Pf_t \otimes D_{(f,f_t)} \vol(g) \big)^\vert \\&=
Tf.(\nabla_{\p_t} Pf_t)^\top \otimes \vol(g) + Tf.(Pf_t)^\top \otimes D_{(f,f_t)} \vol(g) \\&=
-Tf.\g(Pf_t,\nabla f_t)^\sharp \otimes \vol(g) + 0 ,
\end{align*}
which is exactly the vertical part of the geodesic equation.

\section{Geodesic distance on shape space}\label{ge}

It came as a big surprise when it was discovered in \cite{Michor98} 
that the Sobolev metric of order zero induces vanishing geodesic distance on shape space $B_i$. 
It will be shown that this problem can be overcome by using higher order Sobolev metrics. 
The proof of this result is based on bounding the $G^P$-length of a path from below by its area swept out. 
The main result is in section~\ref{ge:no}.
The same ideas are contained in
\cite[section~2.4]{Bauer2010}, \cite[section~7]{Michor118} and \cite[section~3]{Michor102}.

\subsection{Geodesic distance on shape space}\label{ge:ge}

\emph{Geodesic distance} on $B_i$ is given by
$$\dist_{G^P}^{B_i}(F_0,F_1) = \inf_F L_{G^P}^{B_i}(F),$$
where the infimum is taken over all $F :[0,1] \to B_i$ with $F(0)=F_0$ and $F(1)=F_1$.
$L_{G^P}^{B_i}$ is the length of paths in $B_i$ given by
$$L_{G^P}^{B_i}(F) = \int_0^1 \sqrt{G^P_F(F_t,F_t)} dt \quad \text{for $F:[0,1] \to B_i$.}$$
Letting $\pi:\Imm \to B_i$ denote the projection, one has 
$$L_{G^P}^{B_i}(\pi \o f) =  L_{G^P}^{\Imm}(f) =\int_0^1 \sqrt{G^P_f(f_t,f_t)} dt$$
when $f:[0,1]\to \Imm$ is horizontal.
In the following sections, conditions on the metric $G^P$ ensuring that $\dist_{G^P}^{B_i}$ separates
points in $B_i$ will be developed.

\subsection{Vanishing geodesic distance}\label{ge:va}

\begin{thm*}
The distance $\dist_{H^0}^{B_i}$ induced by the Sobolev $L^2$ metric of order zero vanishes. 
Indeed it is possible to connect any two distinct shapes by a path of arbitrarily short length. 
\end{thm*}

This result was first established by Michor and Mumford 
for the case of planar curves in \cite{Michor98}.
A more general version can be found in \cite{Michor102}, where the same result 
is proven also on diffeomorphism groups.

\subsection{Area swept out}\label{ge:ar}

For a path of immersions $f$ seen as a mapping $f:[0,1] \x M \to N$ one has
$$(\text{area swept out by $f$})=\int_{[0,1]\x M} \vol(f(\cdot,\cdot)^* \g) 
=\int_0^1 \int_M \norm{f_t^\bot} \vol(g) dt.$$

\subsection{Area swept out bound}\label{ge:ar1}

\begin{lem*}
Let $G^P$ be a Sobolev type metric that is at least as strong as the $H^0$-metric, i.e. 
there is a constant $C_1 > 0$ such that
\begin{align*}
\norm{h}_{G^P} \geq C_1 \norm{h}_{H^0} = C_1 \sqrt{\int_M \g(h,h) \vol(g)} \qquad \text{for all $h \in T\Imm$.} 
\end{align*}
Then one has the area swept out bound for any path of immersions $f$:
\begin{align*}
C_1 \ (\text{area swept out by $f$})
\leq \max_t \sqrt{\Vol\big(f(t)\big)} . L_{G^P}^{\Imm}(f).
\end{align*}
\end{lem*}

The proof is an adaptation of the one given in \cite[section~7.3]{Michor118} for almost local metrics.

\begin{proof}
\begin{align*}
L_{G^P}^{\Imm}(f)&=\int_0^1 \norm{f_t}_{G^P} dt \geq
C_1 \int_0^1 \norm{f_t}_{H^0} dt \\&\geq 
C_1 \int_0^1 \norm{f_t^\bot}_{H^0} dt =
C_1 \int_0^1 \Big(\int_M \norm{f_t^\bot}^2 \vol(g) \Big)^{\frac12} dt \\&\geq
C_1 \int_0^1 \Big(\int_M \vol(g) \Big)^{-\frac12} \int_M 1.\norm{f_t^\bot} \vol(g) dt\\&\geq
C_1 \min_t \Big(\int_M \vol(g) \Big)^{-\frac12} \int_{[0,1]\x M} \vol(f(\cdot,\cdot)^* \g)  \\&=
C_1 \Big(\max_t \int_M \vol(g) \Big)^{-\frac12}\ (\text{area swept out by $f$}) 
\qedhere
\end{align*}
\end{proof}

\subsection{Lipschitz continuity of $\sqrt{\Vol}$}\label{ge:li}

\begin{lem*}
Let $G^P$ be a Sobolev type metric that is at least as strong as the $H^1$-metric, i.e. 
there is a constant $C_2 > 0$ such that
\begin{align*}
\norm{h}_{G^P} \geq C_2 \norm{h}_{H^1} = 
C_2 \sqrt{\int_M \g\big( (1+\Delta) h,h \big) \vol(g)} \qquad \text{for all $h \in T\Imm$.} 
\end{align*}
Then the mapping
$$\sqrt{\Vol}:(B_i,\dist_{G^P}^{B_i}) \to \R_{\geq 0}$$
is Lipschitz continuous, i.e. for all $F_0$ and $F_1$ in $B_i$ one has:
$$
\sqrt{\Vol(F_1)}-\sqrt{\Vol(F_0)} 
\leq \frac{1}{2 C_2} \dist_{G^P}^{B_i}(F_0,F_1).
$$
\end{lem*}

For the case of planar curves, this has been proven in  \cite[section~4.7]{Michor107}.

\begin{proof}
\begin{align*}
\p_t \Vol &= \int_M \Big(\on{div}^g(f_t^\top)-\g\big(f_t^\bot, \Tr^g(S)\big)\Big) \vol(g) \\&=
0+\int_M \g(f_t, \nabla^* Tf) \vol(g) =
\int_M (g^0_1  \otimes \g)(\nabla f_t, Tf) \vol(g) \\&\leq
\sqrt{\int_M  \norm{\nabla f_t}_{g^0_1  \otimes \g}^2 \vol(g)} 
\sqrt{\int_M  \norm{Tf}_{g^0_1  \otimes \g}^2 \vol(g)} \\&\leq
\norm{f_t}_{H^1}\ \sqrt{\Vol}  \leq \frac{1}{C_2}  \norm{f_t}_{G^P}\ \sqrt{\Vol} .
\end{align*}
Thus
\begin{align*}
\p_t \sqrt{\Vol(f)}=\frac{\p_t \Vol(f)}{2 \sqrt{\Vol(f)}}\leq
\frac{1}{2 C_2} \norm{f_t}_{G^P}.
\end{align*}
By integration one gets
\begin{align*}
\sqrt{\Vol(f_1)}-\sqrt{\Vol(f_0)} &= 
\int_0^1 \p_t \sqrt{\Vol(f)}dt \leq 
\int_0^1 \frac{1}{2 C_2} \norm{f_t}_{G^P} = 
\frac{1}{2 C_2}\ L_{G^P}^{\Imm}(f).
\end{align*}
Now the infimum over all paths $f:[0,1] \rightarrow \Imm$ with $\pi(f(0))=F_0$ and $\pi(f(1))=F_1$ is taken. 
\end{proof}

\subsection{Non-vanishing geodesic distance}\label{ge:no}

Using the estimates proven above and the fact that the area swept out separates points at least on $B_e$, 
one gets the following result:

\begin{thm*}
The Sobolev type metric $G^P$ induces non-vanishing geodesic distance on $B_e$ if it is
stronger or as strong as the $H^1$-metric, i.e. if there is a constant $C > 0$ such that
\begin{align*}
\norm{h}_{G^P} \geq C \norm{h}_{H^1} = 
C \sqrt{\int_M \g\big( (1+\Delta) h,h \big) \vol(g)} \qquad \text{for all $h \in T\Imm$.} 
\end{align*}
\end{thm*}

\begin{proof}
By lemma \ref{ge:ar1} we have 
\begin{align*}
C_1 \ (\text{area swept out by $f$})
\leq \max_t \sqrt{\Vol\big(f(t)\big)} . L_{G^P}^{\Imm}(f).
\end{align*}
Now we use the Lipschitz continuity \ref{ge:li} of $\sqrt{\Vol}$ and that area swept out separates
points on $B_e$. 
\end{proof}

\section[Sobolev metrics induced by the Laplacian]{Sobolev metrics induced by the Laplace operator}\label{la}

The results on non-vanishing geodesic distance from the previous section lead us to consider 
operators $P$ that are induced by the Laplacian operator: 
$$P=1+A \Delta^p, \quad P \in \Ga\big(L(T\Imm;T\Imm)\big)$$
for a constant $A>0$. 
(See section~\ref{no:la} for the definition of the Laplacian that is used in this work.)
At every $f \in \Imm$, $P_f$ is a positive, selfadjoint and bijective operator
of order $2p$ acting on $T_f\Imm = \Ga(f^*TN)$.
Note that $\Delta$ depends smoothly on the immersion $f$ via the pullback-metric $f^*\g$, 
so that the same is true of $P$.
$P$ is invariant under the action of the reparametrization group $\on{Diff}(M)$.
It induces the Sobolev metric
$$G_f^P(h,k)=\int_M \g\big(P_f (h),k\big) \vol(g)
=\int_M \g\Big(\big(1+A (\De^{f^*\g})^p \big)h,k\Big) \vol(f^*\g). $$
When $A=1$ we write $H^p := G^{1+\De^p}$.

In this section we will calculate explicitly for $P=1+A \De^p$
the geodesic equation and conserved momenta that have been deduced 
in section~\ref{so} for a general operator $P$. The hardest part will be the 
partial integration needed for the adjoint of $\nabla P$. 
As a result we will get explicit formulas that are ready to be implemented numerically.

\subsection{Other choices for $P$}\label{la:ot}

Other choices for $P$ are the operator $P=1+A (\nabla^*)^p \nabla^p$ 
corresponding to the metric
$$G_f^P(h,k)=\int_M \big(\g(h,k)+A \g(\nabla^p h,\nabla^p k) \big) \vol(g),$$
and other operators that differ only in lower order terms. Since these operators all 
have the same principal symbol, they induce equivalent metrics on each tangent space $T_f \Imm$.
It would be interesting to know if the induced geodesic distances on $B_i$ are equivalent as well.

\subsection{Adjoint of $\nabla P$}\label{la:ad}

To find a formula for the geodesic equation one has to calculate the adjoint of $\nabla P$, 
see section~\ref{so:ge}.
The following calculations at the same time show the existence of the adjoint.
It has been shown in section~\ref{so:ad}
that the invariance of the operator $P$ with respect to reparametrizations
determines the tangential part of the adjoint:
\begin{align*} 
\adj{\nabla P}(h,k)\big)^\top &=\grad^g \g(Ph,k)-\big(\g(Ph,\nabla k)+\g(\nabla h,Pk)\big)^\sharp.
\end{align*} 
It remains to calculate its normal part using the variational formulas from section~\ref{va}.

In the following calculations there will be terms of the form $\Tr(g\i s_1g\i s_2)$, 
where $s_1,s_2$ are two-forms on $M$. 
When the two-forms are seen as mappings $TM \to T^*M$, they can be composed with $g\i:T^*M \to TM$. 
Thus the expression under the trace is a mapping $TM \to TM$ to which the trace can be applied. 
When one of the two-forms is vector valued, the same tensor components as before are contracted.
For example when $h \in \Ga(f^*TN)$ then $s_2=\nabla^2 h$ is a two-form on $M$ with values in $f^*TN$. 
Then in the expression $\Tr(g\i.s_1.g\i.s_2)$ only $TM$ and $T^*M$ components are contracted, 
whereas the $f^*TN$ component remains unaffected.

{\allowdisplaybreaks
\begin{align*}
&\int_M \g\big(m^\bot,\adj{\nabla P}(h,k)\big) \vol(g)=
\int_M \g\big((\nabla_{m^\bot} P)h,k\big) \vol(g)\\
&\quad=A\sum_{i=0}^{p-1}\int_M\g((\nabla_{m^\bot}\Delta)\Delta^{p-i-1}h ,\Delta^{i}k )\vol(g)\\
&\quad =A\sum_{i=0}^{p-1}\int_M\g\Big(\Tr\Big(g\i.D_{(f,m^\bot)}g.g\i\nabla^2\Delta^{p-i-1}h\Big) ,\Delta^{i}k \Big)\vol(g)\\
&\qquad\qquad -A\sum_{i=0}^{p-1}\int_M\g\Big(\nabla_{\big(\nabla^*(D_{(f,m^\bot)}g)+\frac12 
d\Tr^g(D_{(f,m^\bot)}g)\big)^\sharp}\Delta^{p-i-1}h ,\Delta^{i}k \Big)\vol(g)\\
&\qquad\qquad+A\sum_{i=0}^{p-1}\int_M\g\Big(\nabla^*R^{\g}(m^\bot,Tf)\Delta^{p-i-1}h ,\Delta^{i}k \Big)\vol(g)\\
&\qquad\qquad-A\sum_{i=0}^{p-1}\int_M\g\Big(\Tr^g\big(R^{\g}(m^\bot,Tf)\nabla\Delta^{p-i-1}h\big) ,\Delta^{i}k \Big)\vol(g)\\
&=
A\sum_{i=0}^{p-1}\int_M\Tr\Big(g\i.D_{(f,m^\bot)}g.g\i \g(\nabla^2\Delta^{p-i-1}h,\Delta^{i}k )\Big) \vol(g)\\
&\qquad\qquad -A\sum_{i=0}^{p-1}\int_M (g^0_1\otimes \g)\Big(\nabla\Delta^{p-i-1}h ,(\nabla^*D_{(f,m^\bot)}g)\otimes\Delta^{i}k \Big)\vol(g)\\
&\qquad\qquad -A\sum_{i=0}^{p-1}\int_M (g^0_1\otimes \g)\Big(\nabla\Delta^{p-i-1}h ,\frac12 d\Tr^g(D_{(f,m^\bot)}g)\otimes\Delta^{i}k \Big)\vol(g)\\
&\qquad\qquad+A\sum_{i=0}^{p-1}\int_M(g^0_1\otimes \g)\Big(R^{\g}(m^\bot,Tf)\Delta^{p-i-1}h ,\nabla\Delta^{i}k \Big)\vol(g)\\
&\qquad\qquad-A\sum_{i=0}^{p-1}\int_M\g\Big(\Tr^g\big(R^{\g}(m^\bot,Tf)\nabla\Delta^{p-i-1}h\big) ,\Delta^{i}k \Big)\vol(g)
\end{align*}
Using the following  symmetry property of the curvature tensor (see \cite[24.4.4]{MichorH}):
$$\g(R^{\g}(X,Y)Z,U)=-\g(R^{\g}(Y,X)Z,U)=-\g(R^{\g}(Z,U)Y,X)$$
yields:
\begin{align*}
&\int_M \g\big(m^\bot,\adj{\nabla P}(h,k)\big) \vol(g)=\\
&\qquad=
A\sum_{i=0}^{p-1}\int_Mg^0_2\Big(D_{(f,m^\bot)}g,\g(\nabla^2\Delta^{p-i-1}h,\Delta^{i}k )\Big) \vol(g)\\
&\qquad\qquad -A\sum_{i=0}^{p-1}\int_M g^0_1\Big(\g(\nabla\Delta^{p-i-1}h,\Delta^{i}k),\nabla^*D_{(f,m^\bot)}g \Big)\vol(g)\\
&\qquad\qquad -A\sum_{i=0}^{p-1}\int_M g^0_1\Big(\g(\nabla\Delta^{p-i-1}h,\Delta^{i}k) ,\frac12 \nabla\Tr^g(D_{(f,m^\bot)}g) \Big)\vol(g)\\
&\qquad\qquad+A\sum_{i=0}^{p-1}\int_M\g\Big(\Tr^g\big(R^{\g}(\Delta^{p-i-1}h,\nabla\Delta^{i}k)Tf\big) ,m^\bot \Big)\vol(g)\\
&\qquad\qquad-A\sum_{i=0}^{p-1}\int_M\g\Big(\Tr^g\big(R^{\g}(\nabla\Delta^{p-i-1}h,\Delta^{i}k)Tf\big) , m^\bot\Big)\vol(g)\\
&\qquad=
A\sum_{i=0}^{p-1}\int_Mg^0_2\Big(D_{(f,m^\bot)}g,\g(\nabla^2\Delta^{p-i-1}h,\Delta^{i}k )\Big) \vol(g)\\
&\qquad\qquad -A\sum_{i=0}^{p-1}\int_M g^0_2\Big(\nabla\g(\nabla\Delta^{p-i-1}h,\Delta^{i}k),D_{(f,m^\bot)}g \Big)\vol(g)\\
&\qquad\qquad -\frac{A}{2}\sum_{i=0}^{p-1}\int_M \Big(\nabla^*\g(\nabla\Delta^{p-i-1}h,\Delta^{i}k) \Big)\Tr^g(D_{(f,m^\bot)}g) \vol(g)\\
&\qquad\qquad+A\sum_{i=0}^{p-1}\int_M\g\Big(\Tr^g\big(R^{\g}(\Delta^{p-i-1}h,\nabla\Delta^{i}k)Tf \big),m^\bot \Big)\vol(g)\\
&\qquad\qquad-A\sum_{i=0}^{p-1}\int_M\g\Big(\Tr^g\big(R^{\g}(\nabla\Delta^{p-i-1}h,\Delta^{i}k)Tf\big) , m^\bot\Big)\vol(g)\\
&\qquad=
A\sum_{i=0}^{p-1}\int_Mg^0_2\Big(D_{(f,m^\bot)}g,\g(\nabla^2\Delta^{p-i-1}h,\Delta^{i}k )\Big) \vol(g)\\
&\qquad\qquad -A\sum_{i=0}^{p-1}\int_M g^0_2\Big(\g(\nabla^2\Delta^{p-i-1}h,\Delta^{i}k),D_{(f,m^\bot)}g \Big)\vol(g)\\
&\qquad\qquad -A\sum_{i=0}^{p-1}\int_M g^0_2\Big(\g(\nabla\Delta^{p-i-1}h,\nabla\Delta^{i}k),D_{(f,m^\bot)}g \Big)\vol(g)\\
&\qquad\qquad -\frac{A}{2}\sum_{i=0}^{p-1}\int_M \Big(\nabla^*\g(\nabla\Delta^{p-i-1}h,\Delta^{i}k) \Big)\Tr^g(D_{(f,m^\bot)}g) \vol(g)\\
&\qquad\qquad+A\sum_{i=0}^{p-1}\int_M\g\Big(\Tr^g\big(R^{\g}(\Delta^{p-i-1}h,\nabla\Delta^{i}k)Tf \big), m^\bot \Big)\vol(g)\\
&\qquad\qquad-A\sum_{i=0}^{p-1}\int_M\g\Big(\Tr^g\big(R^{\g}(\nabla\Delta^{p-i-1}h,\Delta^{i}k)Tf\big) , m^\bot\Big)\vol(g)\\
&\qquad=
-A\sum_{i=0}^{p-1}\int_Mg^0_2\Big(D_{(f,m^\bot)}g,\g(\nabla\Delta^{p-i-1}h,\nabla\Delta^{i}k )\Big) \vol(g)\\
&\qquad\quad -\frac{A}{2}\sum_{i=0}^{p-1}\int_M 
\Big(\nabla^*\g(\nabla\Delta^{p-i-1}h,\Delta^{i}k) \Big)\Tr^g(D_{(f,m^\bot)}g) \vol(g)\\
&\qquad\qquad+A\sum_{i=0}^{p-1}\int_M\g\Big(\Tr^g\big(R^{\g}(\Delta^{p-i-1}h,\nabla\Delta^{i}k)Tf \big), m^\bot \Big)\vol(g)\\
&\qquad\qquad-A\sum_{i=0}^{p-1}\int_M\g\Big(\Tr^g\big(R^{\g}(\nabla\Delta^{p-i-1}h,\Delta^{i}k)Tf\big) , m^\bot\Big)\vol(g)\\
&\qquad=
-A\sum_{i=0}^{p-1}\int_Mg^0_2\Big(-2.\g(m^\bot,S),\g(\nabla\Delta^{p-i-1}h,\nabla\Delta^{i}k )\Big) \vol(g)\\
&\qquad\qquad -\frac{A}{2}\sum_{i=0}^{p-1}\int_M \Big(\nabla^*\g(\nabla\Delta^{p-i-1}h,\Delta^{i}k) \Big)\Tr^g\big(-2.\g(m^\bot,S)\big) \vol(g)\\
&\qquad\qquad+A\sum_{i=0}^{p-1}\int_M\g\Big(\Tr^g\big(R^{\g}(\Delta^{p-i-1}h,\nabla\Delta^{i}k)Tf \big),m^\bot \Big)\vol(g)\\
&\qquad\qquad-A\sum_{i=0}^{p-1}\int_M\g\Big(\Tr^g\big(R^{\g}(\nabla\Delta^{p-i-1}h,\Delta^{i}k)Tf\big) , m^\bot\Big)\vol(g)
\\&\qquad=
\int_M \g\Big(m^\bot,2A\sum_{i=0}^{p-1}\Tr\big(g\i S g\i \g(\nabla\Delta^{p-i-1}h,\nabla\Delta^{i}k ) \big)\Big)\\&\qquad\qquad
+\int_M \g\Big(m^\bot,A\sum_{i=0}^{p-1} \big(\nabla^*\g(\nabla\Delta^{p-i-1}h,\Delta^{i}k) \big) \Tr^g(S)\Big) \vol(g)\\
&\qquad\qquad+A\sum_{i=0}^{p-1}\int_M\g\Big(\Tr^g\big(R^{\g}(\Delta^{p-i-1}h,\nabla\Delta^{i}k)Tf \big), m^\bot \Big)\vol(g)\\
&\qquad\qquad-A\sum_{i=0}^{p-1}\int_M\g\Big(\Tr^g\big(R^{\g}(\nabla\Delta^{p-i-1}h,\Delta^{i}k)Tf\big) , m^\bot\Big)\vol(g).
\end{align*} 
} 
From this, one can read off the normal part of the adjoint. 
Thus one gets:
\begin{lem*}
The adjoint of $\nabla P$ defined in section~\ref{so:ad} for the operator $P=1+A\De^p$ is
\begin{align*}
\adj{\nabla P}(h,k)&=
2A\sum_{i=0}^{p-1}\Tr\big(g\i S g\i \g(\nabla\Delta^{p-i-1}h,\nabla\Delta^{i}k ) \big)
\\&\qquad
+A\sum_{i=0}^{p-1} \big(\nabla^*\g(\nabla\Delta^{p-i-1}h,\Delta^{i}k) \big) \Tr^g(S)\\
&\qquad+A\sum_{i=0}^{p-1}\Tr^g\big(R^{\g}(\Delta^{p-i-1}h,\nabla\Delta^{i}k)Tf \big)\\
&\qquad-A\sum_{i=0}^{p-1}\Tr^g\big(R^{\g}(\nabla\Delta^{p-i-1}h,\Delta^{i}k)Tf \big)
\\&\qquad
+Tf.\Big[\grad^g \g(Ph,k)-\big(\g(Ph,\nabla k)+\g(\nabla h,Pk)\big)^\sharp\Big].
\end{align*}
\end{lem*}

\subsection{Geodesic equations and conserved momentum}\label{la:ge}

The shortest and most convenient formulation of the geodesic equation is in terms 
of the momentum $p=(1+A\De^p)f_t \otimes \vol(g)$, see sections~\ref{so:gemo} and \ref{so:geshmo}.

\begin{thm*}
The geodesic equation on $\Imm(M,N)$ for the $G^P$-metric 
with $P=1+A \Delta^p$ is given by:
$$\left\{\begin{aligned}
p &= (1+A \De^p)f_t \otimes \vol(g), \\
\nabla_{\p_t}p&=\Bigg(
A\sum_{i=0}^{p-1}\Tr\big(g\i S g\i \g(\nabla(\Delta^{p-i-1}f_t),\nabla\Delta^{i}f_t ) \big)\\&\quad
+\frac{A}{2}\sum_{i=0}^{p-1}\big(\nabla^*\g(\nabla(\Delta^{p-i-1}f_t),\Delta^{i}f_t) \big).\Tr^g(S)\\&\quad
+2A\sum_{i=0}^{p-1}\Tr^g\big(R^{\g}(\Delta^{p-i-1}f_t,\nabla\Delta^{i}f_t)Tf\big)\\&\quad
-\frac12\g(Pf_t,f_t) \Tr^g(S)
-Tf.\g(Pf_t,\nabla f_t)^\sharp\Bigg) \otimes \vol(g).
\end{aligned}\right.$$
This equation is well-posed by theorem \ref{so:we} since all 
conditions are satisfied.
For the special case of plane curves, this agrees with the geodesic equation calculated in 
\cite[section~4.2]{Michor107}.
\end{thm*}

$P=1+A\De^p$ and consequently $G^P$ are invariant under the action of the reparametrization group
$\Diff(M)$. According to section~\ref{so:mo} one gets:
\begin{thm*}
The momentum mapping for the action of $\Diff(M)$ on $\Imm(M,N)$
$$g\Big(\big((1+A\De^p)f_t \big)^\top\Big) \otimes \vol(g)\in \Ga(T^*M\otimes_M\vol(M))$$
is constant along any geodesic $f$ in $\Imm(M,N)$.
\end{thm*}

The horizontal geodesic equation for a general metric on $\Imm$ has been derived in 
section~\ref{sh:geshmo}. In section~\ref{so:geshmo} it has been shown that this equation 
takes a very simple form. Now it is possible to write down this equation specifically for 
the operator $P=1+A\De^p$:
\begin{thm*}
The geodesic equation on shape space for the Sobolev-metric $G^P$ with $P=1+A\De^p$ is equivalent 
to the set of equations
\begin{equation*}
\left\{\begin{aligned}
p &= Pf_t \otimes \vol(g), \qquad Pf_t = (Pf_t)^\bot, \\
(\nabla_{\p_t}p)^\hor &= 
\Bigg(A\sum_{i=0}^{p-1}\Tr\big(g\i S g\i \g(\nabla\Delta^{p-i-1}f_t,\nabla\Delta^{i}f_t ) \big)
\\&\qquad+ \frac{A}{2}\sum_{i=0}^{p-1} \big(\nabla^*\g(\nabla\Delta^{p-i-1}f_t,\Delta^{i}f_t) \big) \Tr^g(S)
\\&\qquad+2A\sum_{i=0}^{p-1}\Tr^g\big(R^{\g}(\Delta^{p-i-1}f_t,\nabla\Delta^{i}f_t)Tf\big)
\\&\qquad-\frac12\g(Pf_t,f_t).\Tr^g(S) \Bigg) \otimes \vol(g),
\end{aligned}\right.
\end{equation*}
where $f$ is a path of immersions. 
For the special case of plane curves, this agrees with the geodesic equation calculated in 
\cite[section~4.6]{Michor107}.
\end{thm*}

\section{Surfaces in $n$-space}\label{su}
 
This section is about the special case where the ambient space $N$ is $\R^n$. 
The flatness of $\R^n$ leads to a simplification of the geodesic equation, and the 
Euclidean motion group acting on $\R^n$ induces additional conserved quantities. 
The vector space structure of $\R^n$ allows to define a 
Fr\'echet metric. 
This metric will be compared to Sobolev metrics. 
Finally in section~\ref{su:co} the space of concentric hyper-spheres in $\R^n$
is briefly investigated.

\subsection{Geodesic equation}\label{su:ge}

The covariant derivative $\nabla^{\g}$ on $\R^n$ is but the usual derivative.
Therefore the covariant derivatives $\nabla_{\p_t} f_t$ and $\nabla_{\p_t}p$ in the geodesic 
equation can be replaced by $f_{tt}$ and $p_t$, respectively.
(Note that $\Imm(M,\R^n)$ is an open subset of the Fr\'echet vector space 
$C^\infty(M,\R^n)$.)
Also, the curvature terms disappear because $\R^n$ is flat.
Any of the formulations of the geodesic equation presented so far can thus be
adapted to the case $N=\R^n$. 

We want to show how the geodesic equation simplifies further under the additional assumptions 
that $\dim(M)=\dim(N)-1$ and that $M$ is orientable. Then it is possible define a unit
vector field $\nu$ to $M$. The condition that $f_t$ is horizontal then simplifies 
to $Pf_t = a.\nu$ for $a \in C^\infty(M)$. 
The geodesic equation can then be written as an equation for $a$. 
However, the equation is slightly simpler 
when it is written as an equation for $a.\vol(g)$. 
In practise, $\vol(g)$ can be treated as a function on $M$ because
one can identify $\vol(g)$ with its density with respect to $du^1 \wedge \ldots \wedge du^{n-1}$, 
where $(u^1, \ldots, u^{n-1})$ is a chart on $M$. Thus multiplication by $\vol(g)$ 
does not pose a problem. 

\begin{thm*}
The geodesic equation for a Sobolev-type metric $G^P$ on shape space 
$B_i(M,\R^n)$ with $\dim(M)=n-1$
is equivalent to the set of equations
\begin{equation*}
\left\{\begin{aligned}
Pf_t  &= a.\nu \\
\p_t\big(a.\vol(g)\big) &= \frac12 \g\big(\adj{\nabla P}(f_t,f_t),\nu\big) 
-\frac12 \g(Pf_t,f_t) \g\big(\Tr^g(S),\nu\big),
\end{aligned}\right.
\end{equation*}
where $f$ is a path in $\Imm(M,\R^n)$ and $a$ is a time-dependent function on $M$. 
\end{thm*}
\begin{proof}
Applying $\g(\cdot,\nu)$ to the geodesic equation \ref{so:geshmo} 
on shape space in terms of the momentum one gets
\begin{align*}
\p_t\big(a.\vol(g)\big) &= 
\p_t\ \g\big(Pf_t \otimes \vol(g),\nu\big) \\&=
\g\Big(\nabla_{\p_t}\big(Pf_t \otimes \vol(g)\big),\nu\Big) 
+ \g\big(Pf_t \otimes \vol(g),\nabla_{\p_t} \nu\big)  \\&=
\frac12 \g\big(\adj{\nabla P}(f_t,f_t),\nu\big) -\frac12 \g(Pf_t,f_t) \g\big(\Tr^g(S),\nu\big)+ 0.
\qedhere
\end{align*}
\end{proof}
Let us spell this equation out in even more details for the $H^1$-metric.
This is the case of interest for the numerical examples in section~\ref{nu}. 
\begin{thm*}
The geodesic equation on shape space $B_i(M,\R^n)$ 
for the Sobolev-metric $G^P$ with $P=1+A\De$ is equivalent 
to the set of equations
\begin{equation*}
\left\{\begin{aligned}
Pf_t &=  a.\nu \\
\p_t \big(a.\vol(g)\big) &= 
\Big(A g^0_2\big(s, \g(\nabla f_t,\nabla f_t ) \big) 
-\frac{\Tr(L)}{2} \big(\norm{f_t}_{\g}^2 + A \norm{\nabla f_t}_{g^0_1\otimes\g}^2 \big) 
\Big) \vol(g),
\end{aligned}\right.
\end{equation*}
where $f$ is a path of immersions, 
$a$ is a time-dependent function on $M$, 
$s=\g(S,\nu) \in \Ga(T^0_2 M)$ is the shape operator,
$L=g\i s \in \Ga(T^1_1 M)$ is the Weingarten mapping, 
and $\Tr(L)$ is the mean curvature.
\end{thm*}
\begin{proof}
The fastest way to get to this equation is to apply $\g(\cdot,\nu)$ to
the geodesic equation on $\Imm$ from section~\ref{la:ge}.
This yields
\begin{align*}
\p_t \big(a.\vol(g)\big) &= 
\Big(A \Tr\big(g\i.s.g\i \g(\nabla f_t,\nabla f_t ) \big)
+ \frac{A}{2} \big(\nabla^*\g(\nabla f_t,f_t) \big) \Tr(L)\\&\qquad
-\frac12\g(Pf_t,f_t).\Tr(L) \Big) \vol(g)\\&=
\Big(A g^0_2\big(s, \g(\nabla f_t,\nabla f_t ) \big)
- \frac{A}{2} \Tr^g\big(\g(\nabla^2 f_t,f_t) \big) \Tr(L)\\&\qquad
- \frac{A}{2} \Tr^g\big(\g(\nabla f_t,\nabla f_t) \big) \Tr(L)
-\frac12\g\big((1+A\De)f_t,f_t\big)\Tr(L) \Big) \vol(g)\\&=
\Big(A g^0_2\big(s, \g(\nabla f_t,\nabla f_t ) \big)
- \frac{A}{2} \Tr^g\big(\g(\nabla f_t,\nabla f_t) \big) \Tr(L)\\&\qquad
-\frac12\g\big(f_t,f_t\big).\Tr(L) \Big) \vol(g).
\end{align*}
Notice that the second order derivatives of $f_t$ have canceled out.
\end{proof}

\subsection{Additional conserved momenta}
 
If $P$ is invariant under the action of the Euclidean motion group $\R^n\rtimes\on{SO}(n)$, 
then also the metric $G^P$ is in invariant under this group action 
and one gets additional conserved quantities as described in \cite[section~2.5]{Michor107}:

\begin{thm*}
For an operator $P$ that is invariant under the action of the 
Euclidean motion group $\R^n\rtimes\on{SO}(n)$, 
the linear momentum
$$ \int_M  Pf_t \vol(g)\in(\R^n)^* $$
and the angular momentum
\begin{align*}
\forall X\in \mathfrak{so}(n): \int_M \g( X.f,Pf_t ) \vol(g) \\
\text{or equivalently } \int_M (f\wedge Pf_t ) 
\vol(g)\in\textstyle{\bigwedge^2}\R^n\cong \mathfrak{so}(n)^* 
\end{align*}
are constant along any geodesic $f$ in $\Imm(M,\R^n)$.
The operator $P=1+A \De^p$ satisfies this property. 
\end{thm*}

\subsection{Fr\'echet distance and Finsler metric}\label{su:fr}

The Fr\'echet distance on shape space $B_i(M,\R^n)$ is defined as
\begin{align*}
\dist_\infty^{B_i}(F_0,F_1) = \inf_{f_0,f_1} \norm{f_0 - f_1}_{L^\infty},
\end{align*}
where the infimum is taken over all $f_0, f_1$ with $\pi(f_0)=F_0, \pi(f_1)=F_1$.
As before, $\pi$ denotes the projection $\pi:\Imm \to B_i$. 
Fixing $f_0$ and $f_1$, one has
\begin{align*}
\dist_\infty^{B_i}\big(\pi(f_0),\pi(f_1)\big) = \inf_{\varphi} \norm{f_0 \o \varphi - f_1}_{L^\infty},
\end{align*}
where the infimum is taken over all $\varphi \in \on{Diff}(M)$.
The Fr\'echet distance is related to the Finsler metric
\begin{align*}
G^\infty : T \Imm(M,\R^n) \rightarrow \R, \qquad h \mapsto \norm{h^\bot}_{L^\infty}.
\end{align*}

\begin{lem*}
The pathlength distance induced by the Finsler metric $G^\infty$ 
provides an upper bound for the Fr\'echet distance:
\begin{align*}
\dist_\infty^{B_i}(F_0,F_1) \leq \dist_{G^\infty}^{B_i}(F_0,F_1) = \inf_f \int_0^1 \norm{f_t}_{G^\infty} dt,
\end{align*}
where the infimum is taken over all paths 
$$f:[0,1] \to \Imm(M,\R^n) \quad \text{with} \quad \pi(f(0))=F_0, \pi(f(1))=F_1.$$
\end{lem*}

\begin{proof}
Since any path $f$ can be reparametrized such that $f_t$ is normal to $f$, one has
$$\inf_f \int_0^1 \norm{f_t^\bot}_{L^\infty} dt = \inf_f \int_0^1 \norm{f_t}_{L^\infty} dt, $$
where the infimum is taken over the same class of paths $f$ as described above. Therefore
\begin{align*}
\dist_\infty^{B_i}(F_0,F_1) &= \inf_f \norm{f(1)-f(0)}_{L^\infty} 
= \inf_f \norm{ \int_0^1 f_t dt}_{L^\infty} 
\leq \inf_f \int_0^1 \norm{f_t}_{L^\infty} dt \\ &
= \inf_f \int_0^1 \norm{f_t^\bot}_{L^\infty} dt
= \dist_{G^\infty}^{B_i}(F_0,F_1). \qedhere
\end{align*}
\end{proof}

It is claimed in \cite[theorem~13]{MennucciYezzi2008} that $d_\infty=\dist_{G^\infty}$. However, 
the proof given there only works on the vector space $C^\infty(M,\R^n)$ and not on $B_i(M,\R^n)$. 
The reason is that convex combinations of immersions are used in the proof, 
but that the space of immersions is not convex.

\subsection{Sobolev versus Fr\'echet distance}\label{su:fr2}

It is a desirable property of any distance on shape space to be stronger than the Fr\'echet
distance. Otherwise, singular points of a shape could move arbitrarily far away
without increasing the distance much. 

As the following result shows, Sobolev metrics of low order do not have this property. 
The authors believe that this property is true
when the order of the metric is high enough, but were not able to prove this.

\begin{lem*}
Let $G^P$ be a metric on $B_i(M,\R^n)$ that is weaker than or at least as weak as a Sobolev $H^p$-metric with 
$p < \frac{\dim(M)}2+1$, i.e.
\begin{align*}
\norm{h}_{G^P} \leq C \norm{h}_{H^p} = C \sqrt{\int_M \g\big( (1+\Delta^p) h,h \big) \vol(g)} 
\qquad \text{for all $h \in T\Imm$.}
\end{align*}
Then the Fr\'echet distance can not be bounded by the $G^P$-distance. 
\end{lem*}

\begin{proof}
It is sufficient to prove the claim for $P=1+\Delta^p$.
Let $f_0$ be a fixed immersion of $M$ into $\R^n$, 
and let $f_1$ be a translation of $f_0$ by a vector $h$ of length $\ell$. 
It will be shown that the $H^p$-distance between $\pi(f_0)$ and $\pi(f_1)$ 
is bounded by a constant $2L$ that does 
not depend on $\ell$, where $\pi$ denotes the projection of $\Imm$ onto $B_i$. 
Then it follows that the $H^p$-distance can not be bounded from below by the Fr\'echet distance, 
and this proves the claim. 

For small $r_0$, one calculates the $H^p$-length of the following path of immersions:
First scale  $f_0$ by a factor $r_0$, then translate it by $h$, and then
scale it again until it has reached $f_1$. 
The following calculation  shows that under the assumption $p<m/2+1$ 
the immersion $f_0$ can be scaled down to  zero in finite $H^p$-pathlength $L$.
Let $r: [0,1] \to [0,1]$ be a function of time with $r(0)=1$ and $r(1)=0$. 
\begin{align*}
L_{\Imm}^{G^P}\big(r.f_0\big)&=\int_0^1\sqrt{\int_M\g\Big(r_t.\big(1+(\Delta^{(r.f_0)^*\g})^p\big)(f_0),r_t.f_0\Big)\vol\big((r.f_0)^*\g\big)}dt\\
&=\int_0^1\sqrt{\int_M r^2_t.\g\Big(\big(1+\frac{1}{r^{2p}}(\Delta^{f_0^*\g})^p\big)(f_0),f_0\Big)r^m\vol\big(f_0^*\g\big)}dt\\
&=\int_1^0\sqrt{\int_M\g\Big(\big(1+\frac{1}{r^{2p}}(\Delta^{f_0^*\g})^p\big)(f_0),f_0\Big)r^m\vol\big(f_0^*\g\big)}dr =: L
\end{align*}
The last integral converges if $\frac{m-2p}{2}<-1$, which holds by assumption. 
Scaling down to $r_0>0$ needs even less effort. 
So one sees that the length of the shrinking and growing part of the path is bounded by $2L$. 

The length of the translation is simply $\ell \sqrt{r_0^m \Vol(f_0)}=O(r^{m/2})$ 
since the Laplacian of the constant vector field vanishes. Therefore 
\begin{equation*}
\dist_{B_i}^{G^P}\big(\pi(f_0),\pi(f_1)\big) \leq \dist_{\Imm}^{G^P}(f_0,f_1) \leq 2L. \qedhere
\end{equation*}
\end{proof}

\subsection{Concentric spheres}\label{su:co}

For a Sobolev type metric $G^P$ that is invariant under the action of the $SO(n)$ on $\R^n$, 
the set of hyper-spheres in $\R^n$ with common center $0$
is a totally geodesic subspace of $B_i(S^{n-1},\R^n)$. 
The reason is that it is the fixed point set 
of the group $SO(n)$ acting on $B_i$ isometrically. 
(One also needs uniqueness of solutions to the geodesic equation to prove that 
the concentric spheres are totally geodesic.)
This section mainly deals with the case $P=1+\De^p$. 

First we want to determine under what conditions the set of concentric spheres is geodesically complete
under the $G^P$-metric. 

\begin{lem*}
The space of concentric spheres is complete with respect to the $G^P$ 
metric with $P=1+A\Delta^p$ iff $p\geq(n+1)/2$. 
\end{lem*}

\begin{proof}
The space is complete
if and only if it is impossible to scale a sphere down to zero or up to 
infinity in finite $G^P$ path-length. 
So let $f$ be a path of concentric spheres. 
It is uniquely described by its radius $r$. Its velocity is $f_t=r_t.\nu$, where $\nu$ 
designates the unit normal vector field. 
One has
$$\g\big(g\i.S,\nu\big)=:L=-\tfrac{1}{r}\on{Id}_{TM}, 
\quad \Tr(L^k)=(-1)^k\tfrac{n-1}{r^k},\quad \Vol=r^{n-1}\tfrac{n\pi^{n/2}}{\Gamma(n/2+1)}.$$
Keep in mind that $r$ and $r_t$ are constant functions on the sphere, so that all 
derivatives of them vanish.
Therefore 
\begin{align*}
\Delta \nu&=\nabla^*(\nabla \nu)=\nabla^*(-Tf.L)=\Tr^g\Big(\nabla(Tf.L)\Big)
\\&=\Tr^g\Big(\nabla(Tf).L\Big)+\Tr^g\Big(Tf.(\nabla L)\Big)\\&=
\Tr(L^2).\nu+ \Tr^g\Big(Tf.\nabla(-\tfrac1r\on{Id}_{TM})\Big)=\frac{n-1}{r^2}.\nu+0
\end{align*}
and
\begin{align*}
Pf_t&=(1+A\Delta^p)(r_t.\nu)=r_t.\left(1+A\frac{(n-1)^p}{r^{2p}}\right).\nu.
\end{align*}
From this it is clear that the path $f$ is horizontal. Therefore its length as a path in $B_i$ 
is the same as its length as a path in $\Imm$. 
One calculates its length as in the proof of \ref{su:fr2}:
\begin{align*}
L^{G^P}_{B_i}(f)&=\int_0^1\sqrt{G_f^P(f_t,f_t)}dt=
\int_0^1\sqrt{\int_M r_t^2.\left(1+A\frac{(n-1)^p}{r^{2p}}\right)\vol(g)}dt\\&
=\int_0^1 |r_t|\sqrt{ \left(1+A\frac{(n-1)^p}{r^{2p}}\right)\frac{n.\pi^{n/2}}{\Gamma(n/2+1)}r^{n-1}}dt\\&
=\sqrt{\frac{n.\pi^{n/2}}{\Gamma(n/2+1)}}\int_{r_0}^{r_1} \sqrt{\left(1+A\frac{(n-1)^p}{r^{2p}}\right)r^{n-1}}dr.
\end{align*}
The integral diverges for $r_1 \to \infty$ since the integrand is greater than $r^{(n-1)/2}$. 
It diverges for $r_0 \to 0$ iff $(n-1-2p)/2 \leq -1$, which is equivalent to $p \geq (n+1)/2$. 
\end{proof}

The geodesic equation within the space of concentric spheres 
reduces to an ODE for the radius that can be read off 
the geodesic equation in section \ref{la:ge}:
\begin{align*}
r_{tt}=-r_t^2\Big(\frac{n-1}{2r}-\frac{p.A.(n-1)^p}{r\big(r^{2p}+A(n-1)^p\big)}\Big).
\end{align*}

\section{Diffeomorphism groups}\label{di}

For $M=N$ the space $\Emb(M,M)$ equals the \emph{diffeomorphism group of $M$}. 
An operator $P \in \Ga\big(L(T\Emb;T\Emb)\big)$ that is invariant under reparametrizations
induces a right-invariant Riemannian metric on this space. 
Thus one gets the geodesic equation for 
right-invariant Sobolev metrics on diffeomorphism groups and well-posedness of this equation.
To the authors knowledge, well-posedness has so far only been shown for the special case 
$M=S^1$ in \cite{Constantin2003}
and for the special case of Sobolev order one metrics in \cite{GayBalmaz2009}.
Theorem~\ref{so:we} establishes this result for arbitrary compact $M$ and
Sobolev metrics of arbitrary order. 

In the decomposition of a vector $h \in T_f\Emb$ into its tangential and normal components 
$h=Tf.h^\top + h^\bot$, the normal part $h^\bot$ vanishes.
Also $S=\nabla Tf$ vanishes. 
Thus the geodesic equation on $\Diff(M)$ in terms of the momentum $p$ is given by (see \ref{so:gemo})
\begin{equation*}
\left\{\begin{aligned}
p &= Pf_t \otimes \vol(g), \\
\nabla_{\p_t}p &=-Tf.\g(Pf_t,\nabla f_t)^\sharp \otimes \vol(g).
\end{aligned}\right.
\end{equation*}
Note that this equation is not right-trivialized, in contrast to 
the equation given in \cite{Arnold1966,Michor102,Michor109}, for example.
The special case of theorem \ref{so:we} now reads as follows:

\begin{thm*} Let $p\ge 1$ and $k>\frac{\dim(M)}2+1$ and let $P$ satisfy assumptions \thetag{1--3} 
of \ref{so:we}.

Then the initial value problem for the geodesic equation 
has unique local solutions in the Sobolev manifold $\on{Diff}^{k+2p}$ of
$H^{k+2p}$-diffeomorphisms. The solutions depend smoothly on $t$ and on the initial
conditions $f(0,\;.\;)$ and $f_t(0,\;.\;)$.
The domain of existence (in $t$) is uniform in $k$ and thus this
also holds in $\on{Diff}(M)$.

Moreover, in each Sobolev completion $\on{Diff}^{k+2p}$, the Riemannian exponential mapping $\exp^{P}$ exists 
and is smooth on a neighborhood of the zero section in the tangent bundle, 
and $(\pi,\exp^{P})$ is a diffeomorphism from a (smaller) neigbourhood of the zero 
section to a neighborhood of the diagonal in 
$\on{Diff}^{k+2p}\x \on{Diff}^{k+2p}$.
All these neighborhoods are uniform in $k>\dim(M)/2+1$ and can be chosen 
$H^{k_0+2p}$-open, for $k_0 > \dim(M)/2+1$. Thus both properties of the exponential 
mapping continue to hold in $\on{Diff}(M)$.
\end{thm*}

\section{Numerical results}\label{nu}

It is of great interest for shape comparison to solve the \emph{boundary value problem} 
for geodesics in shape space. 
When the boundary value problem can be solved, 
then any shape can be encoded as the initial momentum of a geodesic 
starting at a fixed reference shape. 
Since the initial momenta all lie in the same vector space, 
this also opens the way to statistics on shape space. 

There are two approaches to solving the boundary value problem.
In \cite{Michor118} the first approach of minimizing \emph{horizontal path energy} 
over the set of curves in $\Imm$ 
connecting two fixed boundary shapes has been taken. 
This has been done for several almost local metrics. 
For these metrics it is straightforward to calculate the horizontal energy because 
the horizontal bundle equals the normal bundle.  
However, in the case of Sobolev type metrics 
the horizontal energy involves the inverse of a differential operator (see section~\ref{so:ho}), 
which makes this approach much harder.

\begin{figure}[h]
\centering
\includegraphics[width=\textwidth-10pt]{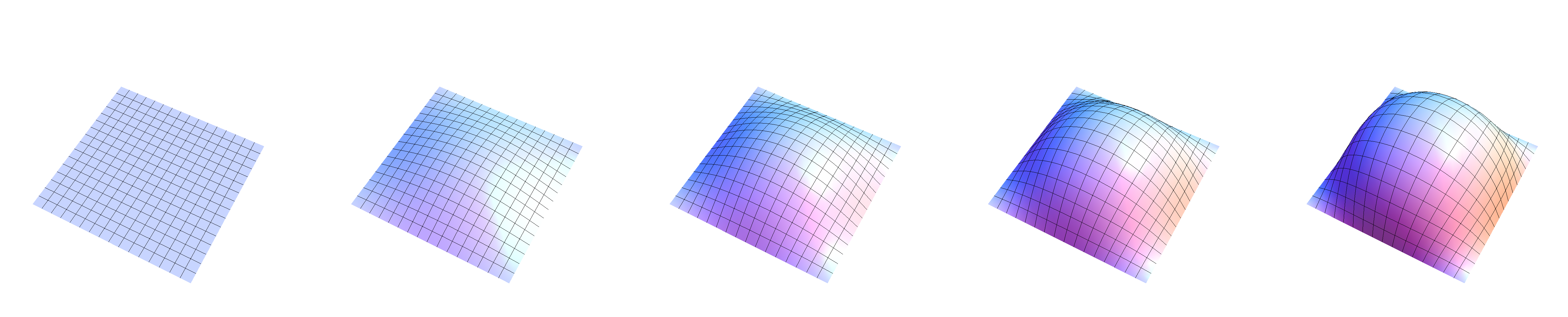}
\caption{Geodesic where a bump is formed out a flat plane. 
The initial momentum is $a=\sin(x)\sin(y)$.
Time increases linearly from left to right. The final time is $t=5$. }
\label{nu:bump_1}
\end{figure}

\begin{figure}[h]
\centering
\begin{psfrags}
\def\PFGstripminus-#1{#1}%
\def\PFGshift(#1,#2)#3{\raisebox{#2}[\height][\depth]{\hbox{%
  \ifdim#1<0pt\kern#1 #3\kern\PFGstripminus#1\else\kern#1 #3\kern-#1\fi}}}%
\providecommand{\PFGstyle}{}%
\psfrag{a08220A}[cr][cr]{\PFGstyle $0.8220$}%
\psfrag{a08220B}[Bc][Bc]{\PFGstyle $0.8220$}%
\psfrag{a08220}[Bc][Bc][1.][0.]{\PFGstyle $0.8220$}%
\psfrag{a08225A}[cr][cr]{\PFGstyle $0.8225$}%
\psfrag{a08225B}[Bc][Bc]{\PFGstyle $0.8225$}%
\psfrag{a08225}[Bc][Bc][1.][0.]{\PFGstyle $0.8225$}%
\psfrag{a08230A}[cr][cr]{\PFGstyle $0.8230$}%
\psfrag{a08230B}[Bc][Bc]{\PFGstyle $0.8230$}%
\psfrag{a08230}[Bc][Bc][1.][0.]{\PFGstyle $0.8230$}%
\psfrag{a08235A}[cr][cr]{\PFGstyle $0.8235$}%
\psfrag{a08235B}[Bc][Bc]{\PFGstyle $0.8235$}%
\psfrag{a08235}[Bc][Bc][1.][0.]{\PFGstyle $0.8235$}%
\psfrag{a08240A}[cr][cr]{\PFGstyle $0.8240$}%
\psfrag{a08240B}[Bc][Bc]{\PFGstyle $0.8240$}%
\psfrag{a08240}[Bc][Bc][1.][0.]{\PFGstyle $0.8240$}%
\psfrag{a08245A}[cr][cr]{\PFGstyle $0.8245$}%
\psfrag{a08245B}[Bc][Bc]{\PFGstyle $0.8245$}%
\psfrag{a08245}[Bc][Bc][1.][0.]{\PFGstyle $0.8245$}%
\psfrag{a08250A}[cr][cr]{\PFGstyle $0.8250$}%
\psfrag{a08250B}[Bc][Bc]{\PFGstyle $0.8250$}%
\psfrag{a08250}[Bc][Bc][1.][0.]{\PFGstyle $0.8250$}%
\psfrag{a0A}[Bc][Bc]{\PFGstyle $0$}%
\psfrag{a0}[Bc][Bc][1.][0.]{\PFGstyle $0$}%
\psfrag{a1A}[tc][tc]{\PFGstyle $1$}%
\psfrag{a1B}[Bc][Bc]{\PFGstyle $1$}%
\psfrag{a1}[Bc][Bc][1.][0.]{\PFGstyle $1$}%
\psfrag{a2A}[tc][tc]{\PFGstyle $2$}%
\psfrag{a2B}[Bc][Bc]{\PFGstyle $2$}%
\psfrag{a2}[Bc][Bc][1.][0.]{\PFGstyle $2$}%
\psfrag{a3A}[tc][tc]{\PFGstyle $3$}%
\psfrag{a3B}[Bc][Bc]{\PFGstyle $3$}%
\psfrag{a3}[Bc][Bc][1.][0.]{\PFGstyle $3$}%
\psfrag{a4A}[tc][tc]{\PFGstyle $4$}%
\psfrag{a4B}[Bc][Bc]{\PFGstyle $4$}%
\psfrag{a4}[Bc][Bc][1.][0.]{\PFGstyle $4$}%
\psfrag{a5A}[tc][tc]{\PFGstyle $5$}%
\psfrag{a5B}[Bc][Bc]{\PFGstyle $5$}%
\psfrag{a5}[Bc][Bc][1.][0.]{\PFGstyle $5$}%
\psfrag{eA}[Bc][Bc]{\PFGstyle $G^P(f_t,f_t)$}%
\psfrag{e}[bc][bc]{\PFGstyle $G^P(f_t,f_t)$}%
\psfrag{tA}[Bc][Bc]{\PFGstyle $t$}%
\psfrag{t}[cl][cl]{\PFGstyle $t$}%
\includegraphics[width=\textwidth-10pt]{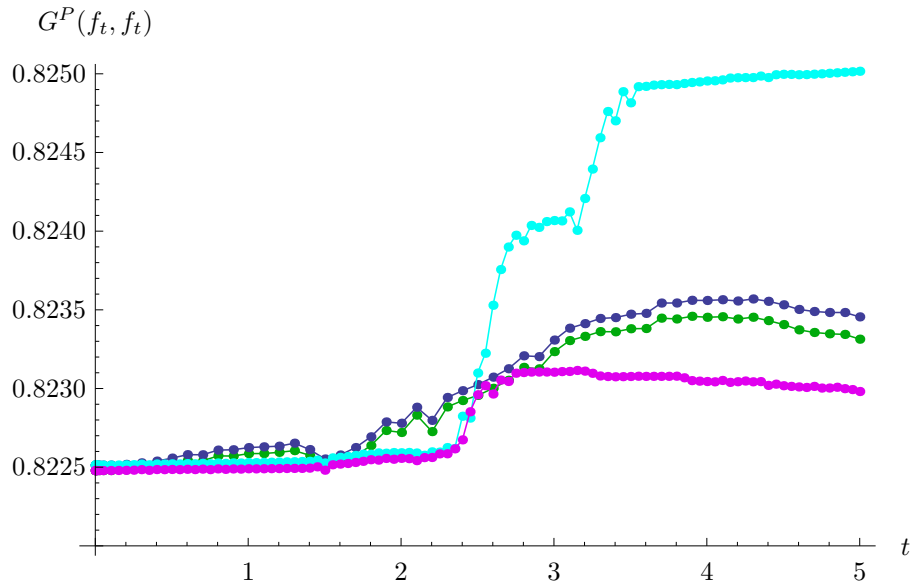}
\end{psfrags}
\caption{Conservation of the energy $G^P(f_t,f_t)$ along the geodesic in figure~\ref{nu:bump_1}. 
The true value of $G^P(f_t,f_t)$ is $\tfrac{\pi^2}{4(1+2A)} \approx 0.822467$ for $A=1$. 
The maximum time step used in blue and green is 0.1. For purple and cyan it is 0.05. The number
of grid points used in blue and cyan is 100 times 100. For green and purple it is 200 times 200.}
\label{nu:energyplot}
\end{figure}

\begin{figure}[h]
\centering
\includegraphics[width=\textwidth-10pt]{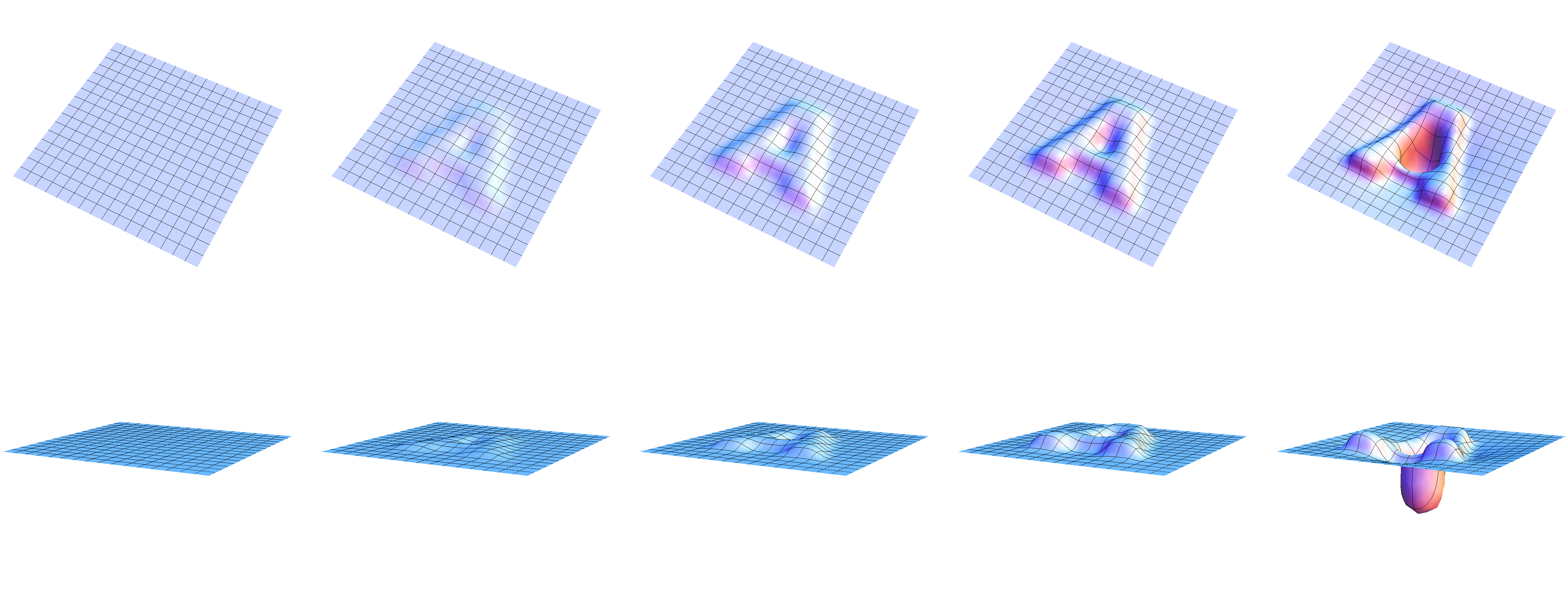}
\caption{Letter A forming along a geodesic path. Time increases linearly from left to right. 
The final time is $t=0.8$.
Top and bottom row are different views of the same geodesic. }
\label{nu:A_ge}
\end{figure}

\begin{figure}[h]
\centering
\includegraphics[width=\textwidth-10pt]{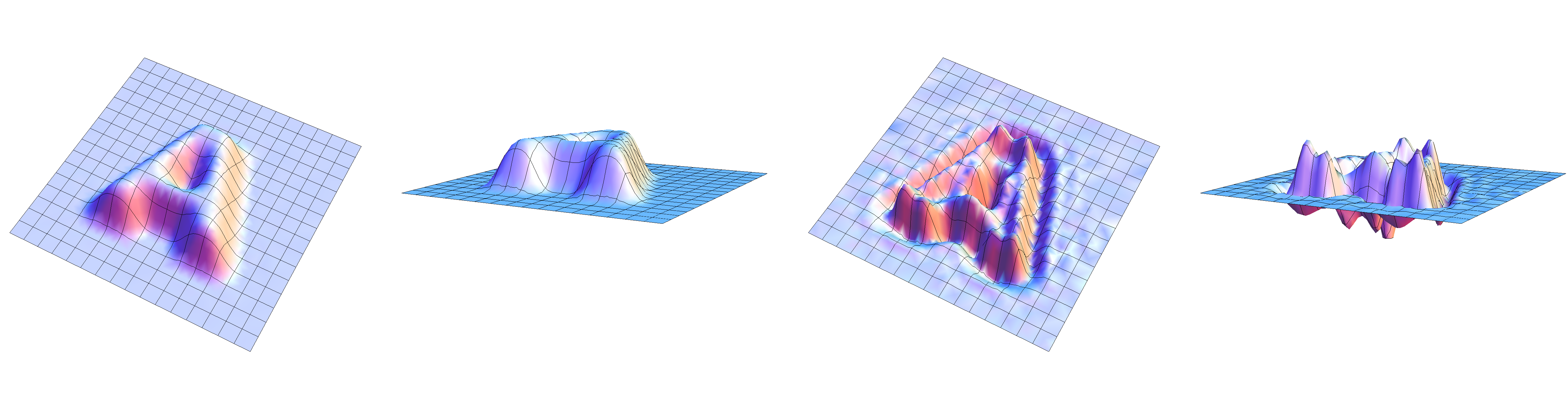}
\caption{Initial velocity $f_t(0,\cdot)$ and momentum $a(0,\cdot)$ of the geodesic in figure \ref{nu:A_ge}. 
Both are shown first from above, then from the side. }
\label{nu:A_in}
\end{figure}

\begin{figure}[h]
\centering
\includegraphics[width=\textwidth-10pt]{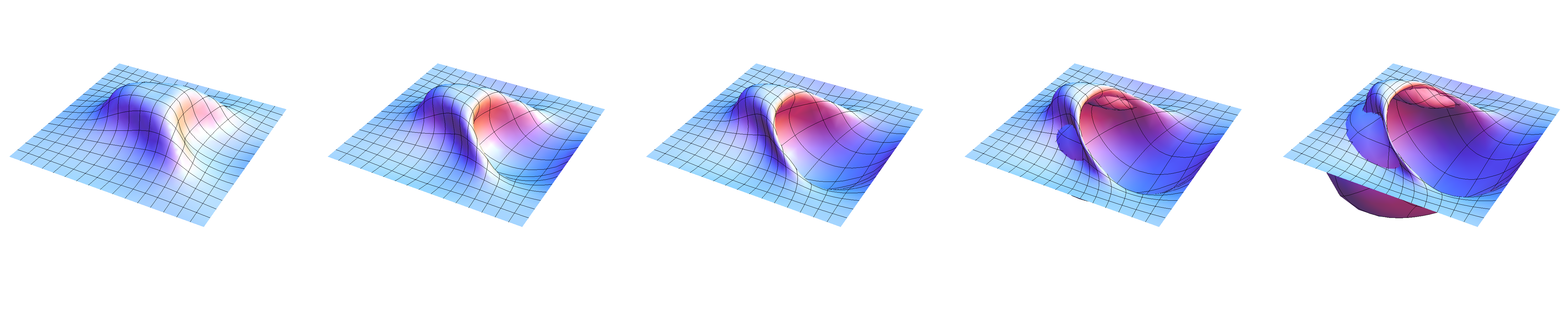}
\caption{A self-intersection forming along a geodesic. Time increases linearly from left to right. }
\label{nu:se}
\end{figure}

The second approach is the method of \emph{geodesic shooting}. 
This method is based on iteratively solving the initial value problem 
while suitably adapting the initial conditions. 
The theoretical requirements of existence of solutions to the geodesic equation and smooth dependence on 
initial conditions are met for Sobolev type metrics, see section \ref{so:we}.
This makes geodesic shooting a promising approach.

{\it The first step towards this aim is to numerically solve the initial value problem for geodesics, 
at least for the $H^1$-metric and the case of surfaces in $\R^3$, 
and that is what will be presented in this work. 
}

The geodesic equation on shape space is equivalent to the horizontal geodesic equation
on the space of immersions. For the case of surfaces in $\R^3$, 
it takes the form given in section~\ref{su:ge}.
This equation can be conveniently set up
using the DifferentialGeometry package incorporated in the computer algebra system Maple
as demonstrated in figure~\ref{nu:maple}. 
(The equations that have actually been solved were simplified 
by multiplying intermediate terms with suitable powers of 
$\sqrt{\vol(g)}$, but for the sake of clearness this has not been included in the Maple code
in figure~\ref{nu:maple}.)

\begin{figure}[h]
\lstset{basicstyle=\small\ttfamily,
	backgroundcolor=\color[gray]{0.9},
	numbers=left, numberstyle=\scriptsize, stepnumber=2, numbersep=5pt,
  	commentstyle=\small,columns=flexible,
  	showstringspaces=false}
\begin{lstlisting}
with(DifferentialGeometry);with(Tensor);with(Tools);
DGsetup([u,v],[x,y,z],E);
f := evalDG(f1(t,u,v)*D_x+f2(t,u,v)*D_y+f3(t,u,v)*D_z);
G := evalDG(dx &t dx + dy &t dy + dz &t dz);
Gamma_vrt := 0 &mult Connection(dx &t D_x &t du);
Tf := CovariantDerivative(f,Gamma_vrt);
g:=ContractIndices(G &t Tf &t Tf,[[1,3],[2,5]]);
g_inv:=InverseMetric(g);
Gamma_bas:=Christoffel(g);
det:=Hook([D_u,D_u],g)*Hook([D_v,D_v],g)-Hook([D_u,D_v],g)^2;
ft := evalDG(diff(f1(t,u,v),t)*D_x+diff(f2(t,u,v),t)*D_y
	+diff(f3(t,u,v),t)*D_z);
ft := convert(ft,DGtensor);
S := CovariantDerivative(Tf,Gamma_vrt,Gamma_bas);
cross:=evalDG((dy &w dz) &t D_x + (dz &w dx) &t D_y + (dx &w dy) &t D_z);
N:=Hook([ContractIndices(Tf &t D_u,[[2,3]]),
	ContractIndices(Tf &t D_v,[[2,3]])],cross);
nu:=convert(N/sqrt(Hook([N,N],G)),DGvector);
s := ContractIndices(G &t S &t nu, [[1,3],[2,6]]);
L := ContractIndices(g_inv &t s,[[2,3]]);
Gftft := ContractIndices(G &t ft &t ft,[[1,3],[2,4]]);
Cft := CovariantDerivative(ft,Gamma_vrt);
CCft := CovariantDerivative(Cft,Gamma_vrt,Gamma_bas);
Dft := ContractIndices(-g_inv &t CCft,[[1,4],[2,5]]);
Pft := evalDG(ft+A*Dft);
GCftCft := ContractIndices(Cft &t Cft &t G,[[1,5],[3,6]]);
gGCftCft := ContractIndices(g_inv &t GCftCft,[[2,3]]);
TrLgGCftCft := ContractIndices(L &t gGCftCft,[[1,4],[2,3]]);
TrgGCftCft := ContractIndices(gGCftCft,[[1,2]]);
TrL := ContractIndices(L,[[1,2]]);
# b(t,u,v) := a(t,u,v)*sqrt(det);
eq1 := ContractIndices(Pft &t dx,[[1,2]])*sqrt(det) = 
	b(t,u,v)*ContractIndices(nu &t dx,[[1,2]]);
eq2 := ContractIndices(Pft &t dy,[[1,2]])*sqrt(det) = 
	b(t,u,v)*ContractIndices(nu &t dy,[[1,2]]);
eq3 := ContractIndices(Pft &t dz,[[1,2]])*sqrt(det) = 
	b(t,u,v)*ContractIndices(nu &t dz,[[1,2]]);
eq4 := diff(b(t,u,v),t) = 
	(A*TrLgGCftCft - TrL/2*(Gftft+A*TrgGCftCft))*sqrt(det);
\end{lstlisting}
\caption{Maple source code to set up the geodesic equation.}
\label{nu:maple}
\end{figure}

Unfortunately, Maple (as of version 14) is not able to solve 
PDEs with more than one space variable numerically. 
Thus the equations were translated into Mathematica.
The PDE was solved using the method of lines. Spatial discretization was done 
using an equidistant grid, and spatial derivatives were 
replaced by finite differences.
The time-derivative $f_t$ appears implicitly in the equation 
$P_f(f_t)=a.\nu$, and this remains so 
when the operator $P_f$ is replaced by finite differences.

The solver that has been used is 
the Implicit Differential-Algebraic 
(IDA) solver that is part of the SUNDIALS suite and is integrated into Mathematica. 
IDA uses backward differentiation of order 1 to 5 with 
variable step widths. Order 5 is standard and has also been used here. 
At each time step, the new value of $f_t$ is computed using some previous values of $f$, and then 
the new value of $f$ is calculated from the equation $P_f( f_t)=a.\nu$.  
The dependence on $f$ in this equation is of course highly nonlinear.
A Newton method is used to solve it. 
This operation is quite costly and has to be done at every step, 
which is a main disadvantage of backward differentiation algorithms.
Explicit methods are probably much better adapted
to the problem. The implementation of an explicit solver 
is ongoing common work of the authors with Martins Bruveris and Colin Cotter.

In the examples that follow, $f$ at time zero is a square $[0,\pi]\x[0,\pi]$ flatly embedded in $\R^3$. 
This is a manifold with boundary, but it can be seen as a part of a bigger closed manifold. 
Zero boundary conditions are used for both $f$ and $a$. 
It remains to specify an initial condition for $a$. 
As a first example, let us assume that $a$ at time zero equals $\sin(x)\sin(y)$, 
where $x,y$ are the Euclidean coordinates on the square.
The resulting geodesic is depicted in figure~\ref{nu:bump_1}. 
In the absence of a closed-form solution of the geodesic equation, one way to 
check if the solution is correct is to see if the energy $G^P(f_t,f_t)$ 
is conserved. Figure~\ref{nu:energyplot} shows 
this for the geodesic from figure~\ref{nu:bump_1}
with various space and time discretizations.

A more complicated example of a geodesic is shown in figure \ref{nu:A_ge} and \ref{nu:A_in}. 
There, the initial velocity was chosen to be a smoothened version of a black and white
picture of the letter A. The initial momentum was computed from it using a discrete 
Fourier transform. 

Finally, it is shown in figure \ref{nu:se} that self-intersections of the surface can occur. 
This is not due to a numerical error but part of the theory, 
and can be an advantage or a disadvantage depending 
on the application.

\bibliographystyle{plain}

\end{document}